\documentclass[11pt,a4paper]{amsart}

\usepackage{graphicx}
\usepackage{amssymb}
\usepackage{amsmath}
\usepackage{stmaryrd}
\usepackage{cite}
\usepackage[mathscr]{eucal}
\usepackage{hyperref}
\usepackage[margin=1.1in]{geometry}
\usepackage{comment}

\usepackage{url}
\usepackage{hyperref}

\title{Complete Ricci solitons via estimates on the soliton potential}
\author{Matthias Wink}
\address{Department of Mathematics, UCLA, 520 Portola Plaza, Los Angeles, CA, 90095}

\email{wink@math.ucla.edu}
\thanks{This work was supported by an EPSRC Research Studentship and the German National Academic Foundation}
\keywords{Ricci solitons, cohomogeneity one}
\subjclass[2010]{53C25, 53C44 (53C30)}

\begin{document}
\newcommand{\diam} {\operatorname{diam}}
\newcommand{\Scal} {\operatorname{Scal}}
\newcommand{\scal} {\operatorname{scal}}
\newcommand{\Ric} {\operatorname{Ric}}
\newcommand{\Hess} {\operatorname{Hess}}
\newcommand{\grad} {\operatorname{grad}}
\newcommand{\Sect} {\operatorname{Sect}}
\newcommand{\Rm} {\operatorname{Rm}}
\newcommand{ \Rmzero } {\mathring{\Rm}}
\newcommand{\Rc} {\operatorname{Rc}}
\newcommand{\Curv} {S_{B}^{2}\left( \mathfrak{so}(n) \right) }
\newcommand{ \tr } {\operatorname{tr}}
\newcommand{ \id } {\operatorname{id}}
\newcommand{ \Riczero } {\mathring{\Ric}}
\newcommand{ \ad } {\operatorname{ad}}
\newcommand{ \Ad } {\operatorname{Ad}}
\newcommand{ \dist } {\operatorname{dist}}
\newcommand{ \rank } {\operatorname{rank}}
\newcommand{\Vol}{\operatorname{Vol}}
\newcommand{\dVol}{\operatorname{dVol}}
\newcommand{ \zitieren }[1]{ \hspace{-3mm} \cite{#1}}
\newcommand{ \pr }{\operatorname{pr}}
\newcommand{\diag}{\operatorname{diag}}
\newcommand{\Lagr}{\mathcal{L}}
\newcommand{\av}{\operatorname{av}}
\newcommand{ \floor }[1]{ \lfloor #1 \rfloor }
\newcommand{ \ceil }[1]{ \lceil #1 \rceil }

\newtheorem{theorem}{Theorem}[section]
\newtheorem{definition}[theorem]{Definition}
\newtheorem{example}[theorem]{Example}
\newtheorem{remark}[theorem]{Remark}
\newtheorem{lemma}[theorem]{Lemma}
\newtheorem{proposition}[theorem]{Proposition}
\newtheorem{corollary}[theorem]{Corollary}
\newtheorem{assumption}[theorem]{Assumption}
\newtheorem{acknowledgment}[theorem]{Acknowledgment}
\newtheorem{DefAndLemma}[theorem]{Definition and lemma}
\newtheorem{questionroman}[theorem]{Question}

\newenvironment{remarkroman}{\begin{remark} \normalfont }{\end{remark}}
\newenvironment{exampleroman}{\begin{example} \normalfont }{\end{example}}
\newenvironment{question}{\begin{questionroman} \normalfont }{\end{questionroman}}

\renewcommand{\labelenumi}{(\alph{enumi})}
\newtheorem{maintheorem}{Theorem}[]
\renewcommand*{\themaintheorem}{\Alph{maintheorem}}
\newtheorem*{theorem*}{Theorem}
\newtheorem*{corollary*}{Corollary}
\newtheorem*{remark*}{Remark}
\newtheorem*{example*}{Example}
\newtheorem*{question*}{Question}

\newcommand{\R}{\mathbb{R}}
\newcommand{\N}{\mathbb{N}}
\newcommand{\Z}{\mathbb{Z}}
\newcommand{\Q}{\mathbb{Q}}
\newcommand{\C}{\mathbb{C}}
\newcommand{\F}{\mathbb{F}}
\newcommand{\X}{\mathcal{X}}
\newcommand{\D}{\mathcal{D}}
\newcommand{\Cont}{\mathcal{C}}

\begin{abstract}
In this paper a growth estimate on the soliton potential is shown for a large class of cohomogeneity one manifolds. This is used to construct continuous families of complete steady and expanding Ricci solitons in the set-ups of L\"u-Page-Pope \cite{LuPagePopeQEinstein} and Dancer-Wang \cite{DWCohomOneSolitons}. It also provides a different approach to the two summands system \cite{WinkSolitonsFromHopfFibrations} which applies to all known geometric examples.
\end{abstract}

\maketitle

\section*{Introduction}

A Ricci soliton is a Riemannian manifold $(M,g)$ together with smooth vector field $X$ on $M$ satisfying the equation
\begin{equation*}
\Ric + \frac{1}{2} L_X g + \frac{\varepsilon}{2} g = 0
\end{equation*}
for some $\varepsilon \in \R,$ where $L_X g$ denotes the Lie derivative of the metric $g$ with respect to $X.$ It is called {\em non-trivial} if $X$ is not a Killing vector field and {\em gradient} if $X$ is the gradient of a smooth function $u \colon M \to \R.$ Furthermore, it is called {\em shrinking}, {\em steady} or {\em expanding} depending on whether $\varepsilon<0,$ $\varepsilon=0$ or $\varepsilon>0.$

Ricci solitons characterise self-similar solutions to the Ricci flow and are therefore natural generalisations of Einstein metrics. One approach to constructing Ricci solitons uses K\"ahler geometry. A Ricci soliton is {\em K\"ahler} if there exists a complex structure on $M$ such that the metric $g$ is K\"ahler and the vector field $X$ is holomorphic. In this case the theory of complex Monge-Amp\`ere equations and its connection to $K$-stability provide tools that have successfully been applied to show the existence of K\"ahler Ricci solitons in various settings, e.g. \cite{WZtoricKRS}, \cite{ISKStabilityToricFanos}, \cite{CSclassifactionKRSonGorensteinDelPezzoSurfaces}.

This paper on the other hand uses the theory of cohomogeneity one manifolds to give constructions of {\em non}-K\"ahler examples. In this case many of the known examples are of warped product type, e.g. \cite{BryantSoliton}, \cite{IveyNewExamplesRS}, \cite{GKExpandingRS}, \cite{DWExpandingSolitons}, \cite{DWSteadySolitons}, \cite{BDGWExpandingSolitons}, \cite{BDWSteadySolitons} or  \cite{AngenentKnopfRSConicalSingNonuniqueness}. In particular, the rotationally symmetric steady Ricci soliton on $\R^n$ for $n \geq 3$, the Bryant soliton \cite{BryantSoliton}, is non-K\"ahler. On the other hand, its $2$-dimensional analogue, Hamilton's cigar \cite{HamiltonRFonSurfaces}, is K\"ahler. Examples of non-K\"ahler steady and expanding Ricci solitons on non-trivial bundles were described in \cite{WinkSolitonsFromHopfFibrations}.

On complex line bundles over $\mathbb{C}P^{n}$ explicit K\"ahler Ricci solitons have been described by Cao \cite{CaoSoliton} and Feldman-Ilmanen-Knopf \cite{FIKSolitons}. More generally, on a complex line bundle over a Fano K\"ahler Einstein manifold, whose Euler class is a rational multiple of the first Chern class of the base, a family of complete non-K\"ahler steady Ricci solitons was constructed by Appleton \cite{AppletonSteadyRS} and Stolarski \cite{StolarskiSteadyRSOnCxLineBundles}. On some of these bundles, Appleton has also found non-collapsed examples. This paper extends their result about the existence of non-K\"ahler Ricci solitons to the expanding case. 

More generally, it will be shown that the examples of steady and expanding K\"ahler Ricci solitons on complex line bundles over {\em products} of Fano K\"ahler Einstein manifolds due to Dancer-Wang \cite{DWCohomOneSolitons} appear amongst continuous families of non-K\"ahler Ricci solitons.

\begin{maintheorem}
Let $L$ be the total space of a complex line bundle over a product of Fano K\"ahler Einstein manifolds with $m \geq 1$ factors whose Euler class is a rational linear combination of the first Chern classes of the base manifolds. Then there exist an $m$-parameter family of complete non-trivial steady and an $(m+1)$-parameter family of complete non-trivial expanding gradient Ricci solitons on $L.$ 
\label{RSOnLineBundlesOverProductsOfFanos}
\end{maintheorem}

The existence of Einstein metrics on the total space $L$ of complex line bundles over Fano K\"ahler Einstein manifolds was discussed in the seminal work of Berard-Bergery \cite{BeBeEinstein}. Generalising this set-up in a different direction, L\"u-Page-Pope \cite{LuPagePopeQEinstein} considered Einstein metrics of warped product type on $S^m \times L$ for $m>1$ and found complete Ricci flat metrics. In fact the sphere can be replaced by an Einstein manifold of positive scalar curvature. All of their examples also support complete steady and expanding Ricci solitons.

\begin{maintheorem}
There exist continuous families of complete non-trivial steady and expanding gradient Ricci solitons on any product $L \times S^m$ of a sphere $S^m,$ $m>1,$ with the total space $L$ of a complex line bundle over a Fano K\"ahler Einstein manifold $M$ whose Euler class is a rational multiple of the first Chern class of $M.$
\label{RSInLuPagePopeSetUp}
\end{maintheorem}

In all the examples above, there is a circle fibre that collapses to a point. The methods of this paper can also be applied if there is a collapsing sphere. As an example, the two summands case of \cite{WinkSolitonsFromHopfFibrations} will be revisited and the results will be extended to a larger parameter range. In particular, all low dimensional examples that are induced by the generalised Hopf fibrations are covered.

\begin{maintheorem} 
There exist a $1$-parameter family of non-homothetic complete steady Ricci solitons and a $2$-parameter family of non-homothetic complete expanding Ricci solitons on the vector bundle associated to the group diagram
\begin{align*}
(G,H,K) = (Sp(m+1), Sp(1) \times Sp(m),  U(1) \times Sp(m))
\end{align*}
for all $m \geq 1.$
\label{TheoremOnLowDimTwoSummands}
\end{maintheorem}
The examples with $m \geq 3$ where already discussed in \cite{WinkSolitonsFromHopfFibrations}, which uses a technique that applies to both Ricci solitons and Einstein metrics. 

In contrast, the constructions in this paper are based on a growth estimate for the soliton potential on cohomogeneity one manifolds with a singular orbit. If the Ricci soliton metric is complete, it is known that in the steady case the potential has linear growth at infinity \cite{BDWSteadySolitons} whereas in the expanding case it has quadratic growth at infinity \cite{BDGWExpandingSolitons}.

However, in the construction problem for Ricci solitons, completeness cannot be assumed a priori. If the metric on the cohomogeneity one manifold remains in diagonal form and the shape operator of the principal orbit is positive definite, which is the case in all currently known examples, then theorem \ref{MainGrowthEstimateTheorem} guarantees {\em immediate linear growth} of the soliton potential with a {\em prescribed growth rate} provided that a condition on the second derivative of the soliton potential at the singular orbit is satisfied. This in turn can be used to construct complete Ricci solitons.

\vspace{2mm}

\textit{Structure of the paper.} The proof of the growth estimate in theorem \ref{MainGrowthEstimateTheorem} as well as the necessary background from the theory of cohomogeneity one Ricci solitons will be explained in section \ref{SectionGrowthSolitonPotential}. Applications to Theorems \ref{RSOnLineBundlesOverProductsOfFanos} - \ref{TheoremOnLowDimTwoSummands} follow in section \ref{ApplicationsSection}.

\vspace{2mm}

\textit{Acknowledgements.} I wish to thank my PhD advisor Andrew Dancer for constant support, helpful comments and numerous discussions, McKenzie Wang for comments on an earlier version of this paper, and Christoph B\"ohm and Burkhard Wilking for remarks that led to proposition \ref{CompletenessForPositivityOfShapeOperator}.

\section{Growth of the soliton potential}
\label{SectionGrowthSolitonPotential}

\subsection{Preliminaries}
\label{PreliminariesSection}

Suppose that $(M^{n+1},g)$ is a Riemannian manifold and $G$ a compact Lie group acting isometrically on $M$ with cohomogeneity one. In the following the existence of a singular orbit will be assumed, and let $K \subset H \subset G$ denote the isotropy groups of the principal and singular orbit, respectively. As every non-trivial steady or expanding Ricci soliton is non-compact, the orbit space $M/G$ will be assumed to be isometric to $[0,T)$ for some $T >0.$ A choice of a unit speed geodesic on $M$ then induces a $G$-equivariant diffeomorphism $\Phi \colon (0,T) \times G/K \to M_0$ onto the union of all principal orbits $M_0$ and the pullback metric is of the form $dt^2 + g_t$ for a $1$-parameter family of metrics $g_t$ on the principal orbit $P=G/K.$ 

The shape operator $L_t$ of the hypersurface $\Phi(\left\lbrace t\right\rbrace \times P)$ will be regarded as a $g_t$-symmetric endomorphism of $TP$ which satisfies $\dot{g}_t = 2 g_t L_t$ and $r_t$ is the Ricci endomorphism corresponding to $g_t.$ Suppose that $u$ is a $G$-invariant function on $M.$ It follows from results of Dancer-Wang \cite{DWCohomOneSolitons} that, due to the existence of a singular orbit, the gradient Ricci soliton equations for $(M,g,u)$ reduce to the ODE system
\begin{align}
- \tr( \dot{L}_t) - \tr(L_t^2) + \ddot{u} + \frac{\varepsilon}{2} & =0, \label{CohomOneRSb} \\
- \dot{L}_t - (- \dot{u} + \tr(L_t)) L_t + r_t + \frac{\varepsilon}{2} \mathbb{I} & = 0. \label{CohomOneRSc}
\end{align}
Conversely, if the metric is at least $C^2$-regular and the soliton potential is of class $C^3,$ then a solution to these equations with $\dot{g}_t = 2 g_t L_t$ induces the structure of a smooth gradient Ricci soliton on $M.$ Furthermore, any $C^3$-regular gradient Ricci soliton has to satisfy the conservation law
\begin{equation}
\ddot{u} + (-\dot{u}+\tr(L)) \dot{u} = C+ \varepsilon u \label{GeneralConservationLaw}
\end{equation}
for some constant $C \in \R.$ It follows from the Ricci soliton equations that the conservation law also takes the form
\begin{equation}
\tr \left( r \right) + \tr ( L^2 ) - \left( - \dot{u} + \tr \left( L \right) \right) ^2 + (n-1) \frac{\varepsilon}{2}  = C +\varepsilon u.
\label{ReformulatedGeneralConsLaw}
\end{equation}

\begin{remarkroman}
Let $\frac{d}{ds} = \frac{1}{- \dot{u} + \tr(L)} \frac{d}{dt}$ and $\mathcal{L} = \frac{1}{-\dot{u}+\tr(L)}.$ Then a direct calculation based only on the Ricci soliton equations shows that 
\begin{align*}
\frac{1}{2} \frac{d}{ds} \frac{C + \varepsilon u}{(-\dot{u}+ \tr(L))^2} = & \ \left( \frac{\tr(L)}{-\dot{u}+ \tr(L)} - 1 \right) \frac{\varepsilon}{2} \mathcal{L}^2  \\
& \ + \frac{C + \varepsilon u}{(-\dot{u}+ \tr(L))^2} \left( \frac{\tr(L^2)}{(-\dot{u}+ \tr(L))^2} - \frac{\varepsilon}{2}  \mathcal{L}^2 \right), \\
\frac{d}{ds} \left( \frac{\tr(L)}{-\dot{u}+ \tr(L)} - 1 \right) = & \ \left( \frac{\tr(L)}{-\dot{u}+ \tr(L)} - 1 \right) \left( \frac{\tr(L^2)}{(-\dot{u}+ \tr(L))^2} - \frac{\varepsilon}{2} \mathcal{L}^2 - 1 \right) \\
& \ + \frac{\tr(L^2)+\tr(r)}{(-\dot{u}+ \tr(L))^2} + (n-1) \frac{\varepsilon}{2} \mathcal{L}^2 - 1.
\end{align*}
The conservation law \eqref{ReformulatedGeneralConsLaw} connects the two equations and identifies two loci inside phase space: The two loci
\begin{align*}
& \left\lbrace \frac{\tr(L)}{-\dot{u}+ \tr(L)} < 1 \ \text{ and } \ \frac{\tr(L^2)+\tr(r)}{(-\dot{u}+ \tr(L))^2} + (n-1) \frac{\varepsilon}{2} \mathcal{L}^2 < 1 \right\rbrace, \\
& \left\lbrace \frac{\tr(L)}{-\dot{u}+ \tr(L)} = 1 \ \text{ and } \ \frac{\tr(L^2)+\tr(r)}{(-\dot{u}+ \tr(L))^2} + (n-1) \frac{\varepsilon}{2} \mathcal{L}^2 = 1 \right\rbrace
\end{align*}
are preserved by the Ricci soliton ODE if $\varepsilon \geq 0.$ The first locus in fact contains all trajectories corresponding to complete steady or expanding Ricci solitons and the second locus corresponds to Einstein trajectories with nonpositive scalar curvature.
\label{PreservedRSandEinsteinLoci}
\end{remarkroman}

After changing $u$ by an additive constant, one may assume that $u(0)=0$. Then the constant $C$ of the conservation law appears as follows in the smoothness conditions of the soliton potential at the singular orbit:
\begin{equation}
u(0)=0, \ \ \ \dot{u}(0)=0, \ \ \ \ddot{u}(0)=\frac{C}{d_S +1}, \label{InitialConditionsPotentialFunction}
\end{equation}
where $d_S$ denotes the dimension of the collapsing sphere $H/K.$ 

It follows from bounds on the scalar curvature due to Chen \cite{ChenStrongUniquenessRF} that any non-trivial complete steady or expanding Ricci soliton has to satisfy $C<0.$ Furthermore, even along a possibly incomplete solution to the steady or expanding Ricci soliton equations with $C<0$ in \eqref{InitialConditionsPotentialFunction}, the soliton potential is strictly decreasing and concave (provided $L_t \neq 0$ if $\varepsilon = 0$), see \cite[Propositions 2.3 and 2.4]{BDWSteadySolitons} and \cite[Proposition 1.11]{BDGWExpandingSolitons}. For the convenience of the reader, and to emphasise that completeness is not required, the proofs will be summarised in a single argument:

\begin{proposition}
Along any Ricci soliton trajectory with $\varepsilon \geq 0$ and $C<0$ in \eqref{InitialConditionsPotentialFunction} that corresponds to a cohomogeneity one manifold of dimension $n+1$ with a singular orbit at $t=0,$ for $t > 0$ and as long as the solution exists, the soliton potential satisfies $u(t)<0,$ $\dot{u}(t)<0$ and also $\ddot{u}(t)<0$ if $\varepsilon >0$ or $\varepsilon = 0$ and $L_t \neq 0.$

Furthermore, if $\varepsilon = 0$ and $C \leq 0,$ there holds $\tr(L_t) \leq \frac{n}{t}$ for $t>0$ and as long as the solution exists.
\label{PotentialFunctionOfExpandingRSAlongCohomOneFlow}
\end{proposition}
\begin{proof}
It follows from the initial conditions \eqref{InitialConditionsPotentialFunction}, that the claim is true for $t>0$ small enough. Suppose that $\dot{u}(t_0)=0$ and $\dot{u}(t)<0$ for $t \in (0,t_0).$ Then it follows that $u(t_0)<0$ and $\ddot{u}(t_0) \geq 0.$ However, the conservation law \eqref{GeneralConservationLaw} implies the contradiction 
\begin{equation*}
0 \leq \ddot{u}(t_0) = C + \varepsilon u(t_0) < 0.
\end{equation*}
Therefore there holds $\dot{u}(t)<0$ for $t>0$ and hence also $u(t)<0$ for $t>0.$ Finally, to prove concavity, differentiate the conservation law \eqref{GeneralConservationLaw} and use \eqref{CohomOneRSb} to get 
\begin{align*}
0 & = \frac{d^3}{dt^3} u + (- \dot{u} + \tr(L)) \ddot{u} - \varepsilon \dot{u} + (- \ddot{u} + \tr( \dot{L} )) \dot{u} \\
  & = \frac{d^3}{dt^3} u + (- \dot{u} + \tr(L)) \ddot{u} - \varepsilon \dot{u} + ( - \tr(L^2) + \frac{\varepsilon}{2} ) \dot{u} \\
  & = \frac{d^3}{dt^3} u + (- \dot{u} + \tr(L)) \ddot{u} - (\frac{\varepsilon}{2} + \tr(L^2) )\dot{u}.
\end{align*}
Therefore, if there is $t_1>0$ such that $\ddot{u}(t_1)=0$ then $\frac{d^3}{dt^3} u (t_1)< 0.$ By continuity this in turn implies that $\ddot{u}(t)<0$ for all $t>0.$

Finally, restrict to the steady case $\varepsilon=0$ and recall that $ \frac{1}{n} \tr(L)^2 \leq \tr(L^2)$ follows from $\tr( (L^{(0)})^2) \geq 0,$ where $L^{(0)}= L - \frac{1}{n} \tr(L) \mathbb{I}$ is the trace-less shape operator. Therefore the Ricci soliton equation \eqref{CohomOneRSb} implies 
\begin{equation*}
\frac{d}{dt} \tr(L) = - \tr(L^2) + \ddot{u} \leq - \frac{1}{n} \tr(L)^2.
\end{equation*}
Note that $f(t) = \left( \frac{1}{f(t_0)} + \frac{1}{n}(t-t_0) \right) ^{-1}$ solves $f^{'} = - \frac{1}{n} f^2.$ Since $\tr(L) = \frac{d_S}{t} + O(t),$ where $d_S$ is the dimension of the collapsing sphere, the claim follows.
\end{proof} 

If the Ricci soliton is {\em complete} then the soliton potential is also necessarily concave in the steady case because the mean curvature of the principal orbit is then always positive. Furthermore the following growth estimates on the potential have been established in \cite[Proposition 2.4]{BDWSteadySolitons} and \cite[Equation (1.10) and Proposition 1.18]{BDGWExpandingSolitons}:

\begin{proposition}
Consider a non-trivial {\em complete} cohomogeneity one gradient Ricci soliton of dimension $n+1$ with a singular orbit at $t=0$ and integrability constant $C<0$ in \eqref{InitialConditionsPotentialFunction}.

In the {\em steady} case it follows that $0<\tr(L_t) \leq \frac{n}{t}$ for all $t > 0$ and 
\begin{align*}
-\dot{u}(t) \to \sqrt{-C} \ \text{ and } \ \ddot{u}(t) \to 0
\end{align*}
as $t \to \infty.$ 

In the {\em expanding} case, there is $t_0>0$ such that $| \tr(L_t)| \leq \sqrt{\frac{n}{2} \varepsilon}$ for all $t \geq t_0$ and given any $\delta \in (0,1)$ there is $t_1>0$ such that 
\begin{align*}
(1- \delta) \frac{\frac{\varepsilon}{2} t_{\text{ }} + \sqrt{n \frac{\varepsilon}{2}} + \sqrt{-C}}{\frac{\varepsilon}{2}t_1+\sqrt{n \frac{\varepsilon}{2}} + \sqrt{-C}}  (-\dot{u}(t_1)) < - \dot{u}(t) < \frac{\varepsilon}{2} t + \sqrt{-C}
\end{align*}
for all $t \geq t_1.$ The second inequality holds in fact for all $t \geq 0.$
\label{AsymptoticsOfPotentialFunction}
\end{proposition}

\subsection{An estimate on the soliton potential}
\label{SectionGrowthEstimateTheorem}

When investigating the Ricci soliton equations in order to construct new examples, the estimates in proposition \ref{AsymptoticsOfPotentialFunction} are {\em not} available because the trajectory may not correspond to a complete Ricci soliton. 

Therefore it is essential that theorem \ref{MainGrowthEstimateTheorem} below does {\em not} assume that the Ricci soliton is complete. Furthermore, it shows that the growth rate of the soliton potential can be prescribed for small times $t>0$, when the existence of a solution to the Ricci soliton equations can be guaranteed by short time existence results, e.g. \cite{BuzanoInitialValueSolitons}. 

\begin{theorem}
Let $(M^{n+1},g,u)$ be a steady or expanding gradient Ricci soliton which is of cohomogeneity one under the isometric action of the compact Lie group $G.$ Suppose $\Phi \colon (0, T) \times P \to M_0$ is an isometric parametrisation of the union of all principal orbits $M_0$ and the pullback metric is given by $dt^2 + g_t.$ Assume that there is a singular orbit $G/H$ at $t=0$ and that the isotropy representation of $P=G/K$ is monotypic, that is, it decomposes into pairwise inequivalent $\Ad_K$-invariant irreducible summands. 

Fix $t_{*} \in [0, T)$, $c_{*}>0$ and suppose that the Killing form $B$ of $\mathfrak{g}$ satisfies $-\tr_{g_{t_{*}}}(B_{|\mathfrak{p}}) < c_{*}$ and that the shape operator $L_t$ of the hypersurface $P_t = \Phi( \lbrace t \rbrace \times P)$ is positive definite for all $t \geq t_{*}.$ 

Then for all $c>0$ and for all $\tau \in (t_{*},T)$ there exists a constant $C_0 \leq 0$ depending only on $c_{*},$ $c,$ $\tau -t_{*}$ such that $- \dot{u}(t) \geq c$ on $[\tau,T),$ provided that $\ddot{u}(0) < C_0.$
\label{MainGrowthEstimateTheorem}
\end{theorem}

Notice that in applications to constructions of Ricci solitons the initial condition $\ddot{u}(0) < C_0$ can simply be imposed.

The proof of theorem \ref{MainGrowthEstimateTheorem} is given in section \ref{ProofGrowthTheoremSection}.

\begin{remarkroman}
The assumption of a monotypic isotropy representation can be relaxed. It suffices that the metric and the Ricci tensor can be diagonalised with respect to the {\em same} decomposition of the isotropy representation for all $t$ as in \eqref{AssumptionOnDecompositionOfMetric}.

It suffices to assume that the shape operator is positive definite on $[t_{*}, \tau]$ and moreover non-vanishing on $[t_{*},T)$ if $\varepsilon = 0.$ Indeed, the proof in section \ref{ProofGrowthTheoremSection} then shows that $- \dot{u}(t) \geq c$ holds at $t = \tau$ and the concavity of the soliton potential in proposition \ref{PotentialFunctionOfExpandingRSAlongCohomOneFlow} ensures that this estimate is preserved for $t \geq \tau.$ 

The theorem even applies to classes of manifolds which are just foliated by hyper-surfaces. These hypersurfaces may not be homogeneous but a similar structure theory has to apply. In particular the metric on the hypersurfaces has to satisfy a decomposition similar to \eqref{AssumptionOnDecompositionOfMetric} so that the Ricci soliton equations have the same form as if they came from a cohomogeneity one manifold.

The assumption of positive definiteness of the shape operator is natural in the sense that it holds in all currently known examples. In fact, propositions \ref{PosDefShapeOperatorNecessaryIfEpsZeroTwoSummands} and  \ref{PosDefShapeOperatorNecessaryIfEpsZeroCircleBundle} show that, in case of the geometries considered in this paper, positive definiteness of the shape operator is also a necessary condition for a maximal solutions of the steady Ricci solitons equations to induce a complete metric.
\label{RemarkOnGrowthEstimate}
\end{remarkroman}

In the applications in section \ref{ApplicationsSection}, the growth estimate will be applied in such a way to ensure that the shape operator indeed remains positive definite along the entire trajectory. This automatically implies completeness of the metric. 

\begin{proposition}
Let $\varepsilon \geq 0$ and let $(g_t, L_t,u)_{t \in (0,T)}$ be a maximal solution to the cohomogeneity one Ricci soliton equations corresponding to a smooth manifold with a singular orbit $Q$ at $t=0,$ metric $g_{M \setminus Q} = dt^2 + g_t$ and $\ddot{u}(0) \leq 0.$ Suppose that the shape operator $L_t$ of the principal hypersurface is positive definite for all $t>0.$

Then $T= \infty$ and the manifold is complete. 
\label{CompletenessForPositivityOfShapeOperator}
\end{proposition}
\begin{proof}
In the steady case proposition \ref{PotentialFunctionOfExpandingRSAlongCohomOneFlow} gives the bound $\tr(L) \leq \frac{n}{t}$ and hence the shape operator remains bounded along the flow. 

In the expanding case, due to the concavity of the soliton potential, the differential inequality 
\begin{align*}
\frac{d}{dt} \tr(L) \leq \frac{\varepsilon}{2} - \tr(L^2) \leq \frac{\varepsilon}{2} - \frac{1}{n} \tr(L)^2
\end{align*}
again shows that $\tr(L)$ and hence $L$ are bounded. Moreover, both arguments carry over to the Einstein case. 

The metric satisfies the {\em linear} ODE $\frac{d}{dt} g_t = 2 g_t L_t$ and hence cannot blow up in finite time either. 

In the soliton case, an estimate on the potential is also required. One may assume that $u(0)=0.$ Recall that the conservation law $\ddot{u} + (-\dot{u}+\tr(L)) \dot{u} = (d_S +1) \ddot{u}(0) + \varepsilon u$ needs to be satisfied, where $d_S$ is the dimension of the collapsing sphere. 

In the steady case, it immediately follows that $\dot{u}^2 \leq - (d_S +1) \ddot{u}(0)$ since $\dot{u}, \ddot{u} \leq 0.$ In the expanding case, concavity of the soliton potential implies $-u(t) \leq - \dot{u}(t) \cdot t.$ It follows that
\begin{align*}
\frac{d}{dt} (- \dot{u}) \leq - (d_S + 1) \ddot{u}(0) + \varepsilon (-\dot{u}) \cdot t - (- \dot{u})^2
\end{align*}
and therefore $- \dot{u} \geq 0$ cannot blow up in finite time. 
\end{proof}

\begin{proposition}
Complete cohomogeneity one steady Ricci solitons with a singular orbit and positive definite shape operator have nonnegative Ricci curvature.
\end{proposition}
\begin{proof}
With the notation of section \ref{PreliminariesSection}, if $N$ is a vector field orthogonal to all principal orbits $TP_t = \Phi \left( \lbrace t \rbrace \times P \right)$, i.e. $N = \Phi_{*} \left( \frac{\partial}{\partial t} \right),$ and $X, Y$ are tangent to all principal orbits, i.e. $X, Y \in TP_{t},$ then the formulas of Eschenburg-Wang \cite{EWInitialValueEinstein} for the Ricci curvature of cohomogeneity one manifolds and the Ricci soliton equations \eqref{CohomOneRSb}, \eqref{CohomOneRSc} show that
\begin{align*}
\Ric(X,Y) & = - g_t(\dot{L}_t(X),Y) - \tr(L_t) g_t(L_t(X),Y) + \Ric_t(X,Y) = - \dot{u}(t) g_t(L_t(X),Y),  \\
\Ric(N,N) & = - \tr( \dot{L}_t ) - \tr( L_t^2 ) =  - \ddot{u}(t).
\end{align*}
By assumption and proposition \ref{PotentialFunctionOfExpandingRSAlongCohomOneFlow}, it follows that $\Ric(X,Y),$ $\Ric(N,N) >0$ for all $t >0.$ Finally, due to results of Dancer-Wang \cite[Proposition 3.19]{DWCohomOneSolitons}, $\Ric(X,N) = 0$ is always satisfied due to the existence of a singular orbit.
\end{proof}

\subsection{Preparation for the proof: Scalar curvature of homogeneous spaces} 

The proof of theorem \ref{MainGrowthEstimateTheorem} involves an analysis of the scalar curvature $s_t=\tr(r_t)$ of the principal hypersurface $P_t.$ This relies on the structure theory for homogeneous spaces, cf. \cite{WZNonExistenceHomEinstein}, \cite{BesseEinstein}, \cite{PSInvariantHomEinstein}, \cite{BohmNonExistenceHomogeneousEinstein}, applied to the principal orbit $P=G/K.$ 

Due to the compactness of $G,$ there exists a biinvariant inner product $b$ on the Lie algebra $\mathfrak{g}$ of $G.$ Furthermore, $\mathfrak{g}$ decomposes into $\Ad_K$-invariant, $b$-orthogonal summands $\mathfrak{g}=\mathfrak{k} \oplus \mathfrak{p}$ and $G$-invariant metrics on $G/K$ correspond to $\Ad_K$-invariant inner products on $\mathfrak{p}.$ 

Let $B(X,Y) = \tr( \ad(X) \circ \ad(Y) )$ be the Killing form of $\mathfrak{g},$ $\langle \cdot, \cdot \rangle$ an $\Ad_K$-invariant inner product on $\mathfrak{p}$ and $e_{\alpha}$ an $\langle \cdot, \cdot \rangle$-orthonormal basis of $\mathfrak{p}.$ Then the scalar curvature $s$ of $G/K$ with respect to $\langle \cdot, \cdot \rangle$ is given by 
\begin{equation*}
s = - \frac{1}{4} \sum_{\alpha, \beta} \langle [ e_{\alpha}, e_{\beta}]_{\mathfrak{p}},  [ e_{\alpha}, e_{\beta}]_{\mathfrak{p}} \rangle - \frac{1}{2} \sum_{\alpha} B(e_{\alpha}, e_{\alpha}),
\end{equation*}
where $[\cdot, \cdot]_{\mathfrak{p}}$ is the projection of the Lie bracket onto $\mathfrak{p}.$ By further decomposing $\mathfrak{p},$ this description can be refined. Suppose that 
\begin{equation}
\mathfrak{p} = \mathfrak{p}_1 \oplus \ldots \oplus \mathfrak{p}_m
\label{DecompositionOfIsoRep}
\end{equation} 
is a decomposition of $\mathfrak{p}$ into $b$-orthogonal, $\Ad_K$-irreducible summands. In fact, for any $\Ad_K$-invariant inner product on $\mathfrak{p}$ there exists such a decomposition which moreover satisfies
\begin{equation*}
\langle \cdot , \cdot \rangle = x_1 b_{|\mathfrak{p}_1} \perp \ldots \perp x_m b_{|\mathfrak{p}_m}
\end{equation*}
for some $x_1, \ldots, x_m >0.$ This decomposition may not be unique. 

Suppose that the $\langle \cdot, \cdot \rangle$-orthonormal basis $\lbrace e_{\alpha} \rbrace$ is adapted to the decomposition \eqref{DecompositionOfIsoRep}. That is, $e_{\alpha} \in \mathfrak{p}_i$ for some $i$ and there holds $\alpha < \beta$ if $i < j.$ Then the coefficients defined by $A_{\alpha \beta}^{\gamma} = b( [ e_{\alpha}, e_{\beta}], e_{\gamma})$ satisfy $[ e_{\alpha}, e_{\beta}]_{\mathfrak{p}} = \sum_{\gamma} A_{\alpha \beta}^{\gamma} e_{\gamma}$ and the structure constants of the decomposition are given by $[i j k] = \sum ( A_{\alpha \beta}^{\gamma} )^2,$ where the sum is taken over all indices $\alpha,$ $\beta,$ $\gamma$ such that $e_{\alpha} \in \mathfrak{p}_i,$ $e_{\beta} \in \mathfrak{p}_j$ and $e_{\gamma} \in \mathfrak{p}_k.$ Notice that the structure constants $[i j k]$ depend on the decomposition of $\mathfrak{p}$ but are independent of the choice of the orthonormal basis. Furthermore, $A_{\alpha \beta}^{\gamma}$ is skew-symmetric in all three indices and hence $[i j k]$ is symmetric in all three indices. 

Let $d_i = \dim \mathfrak{p}_i$ and $B_{| \mathfrak{p}_i } = - b_i b_{|\mathfrak{p}_i}.$ Then it follows that $b_i \geq 0$ with $b_i = 0$ if and only if $\mathfrak{p} _i \subset \mathfrak{z}(\mathfrak{g})$ and $[i j k] \geq 0$ with $[i j k] = 0$ if and only if $b([\mathfrak{p}_i, \mathfrak{p}_j], \mathfrak{p}_k ) =0.$ Furthermore, the scalar curvature of $G/K$ with respect to the inner product $\langle \cdot, \cdot \rangle$ can now be written as 
\begin{equation}
s = - \frac{1}{4} \sum_{i,j,k} [ijk] \frac{x_k}{x_i x_j} + \frac{1}{2} \sum_{i=1}^m \frac{d_i b_i}{x_i}.
\label{ScalarCurvatureOfHomogeneousSpaceB}
\end{equation}
More generally, the Ricci operator is given by
\begin{equation}
r_{| \mathfrak{p}_i} = \left\lbrace \frac{1}{2} \frac{b_i}{x_i} - \frac{1}{2 d_i} \sum_{j,k=1}^m [ijk] \frac{x_k}{x_i x_j} + \frac{1}{4 d_i} \sum_{j,k = 1}^m [ijk] \frac{x_i}{x_j x_k} \right\rbrace  \id_{\mathfrak{p}_i}.
\label{RicciTensorHomogeneousSpace}
\end{equation}

\begin{remarkroman}
Wang-Ziller \cite{WZNonExistenceHomEinstein} have shown the following relation between the structure constants. For any $1 \leq i \leq m$ there holds
\begin{align*}
\sum_{j,k} [ijk] = d_i(b_i-2c_i),
\end{align*}
where the constants $c_i \geq 0$ describe the Casimir operator $\mathcal{C}_{\mathfrak{p}_i, b_{\mathfrak{k}}}.$ By definition one has $\mathcal{C}_{\mathfrak{p}_i, b_{\mathfrak{k}}} = - \sum_{i} \ad(z_i) \circ \ad(z_i)$ for any $b$-orthonormal basis $\lbrace z_i \rbrace$ of $\mathfrak{k}.$ However, due to the irreducibility condition it satisfies $\mathcal{C}_{\mathfrak{p}_i, b_{\mathfrak{k}}} = c_i \id_{\mathfrak{p}_i}.$ Moreover, one has $c_i = 0$ if and only if $\mathfrak{p}_i$ is contained in the subspace $\mathfrak{p}_0$ which satisfies $\Ad_{K  | \mathfrak{p}_0} = \id_{| \mathfrak{p}_0}.$ 
\end{remarkroman}

\subsection{Proof of the growth estimate theorem \ref{MainGrowthEstimateTheorem}}
\label{ProofGrowthTheoremSection} The assumption of a monotypic isotropy representation in theorem \ref{MainGrowthEstimateTheorem} guarantees that the {\em same} decomposition \eqref{DecompositionOfIsoRep} can be chosen to diagonalise the metric on the hypersurface for all $t.$ More generally, it is enough to assume that the metric is given by
\begin{equation}
g_t = f_1^2(t) b_{| \mathfrak{p}_1} + \ldots + f_m^2(t) b_{| \mathfrak{p}_m}
\label{AssumptionOnDecompositionOfMetric}
\end{equation}
for all $t.$ Notice that then the Ricci operator \eqref{RicciTensorHomogeneousSpace} also diagonalises with respect to the same decomposition. Moreover, the shape operator satisfies
\begin{equation*}
L = \diag \left( \frac{\dot{f}_1}{f_1} \id_{\mathfrak{p}_1}, \ldots, \frac{\dot{f}_m}{f_m} \id_{\mathfrak{p}_m} \right) 
\end{equation*}
since $\dot{g}_t = 2 g_t L_t.$ 

One may assume $f_i > 0$ and $u(0)=0.$ Then the shape operator $L$ is positive definite if and only if $\dot{f}_i > 0$ for all $i$ and the smoothness conditions of the soliton potential are given by \eqref{InitialConditionsPotentialFunction}. In particular, the constant $C$ in the conservation law \eqref{GeneralConservationLaw} is given by $C = (d_S+1) \ddot{u}(0).$

Observe that the Ricci soliton equations \eqref{CohomOneRSb}, \eqref{CohomOneRSc} imply
\begin{align*}
\ddot{u} = \tr(L^2) - (- \dot{u} + \tr(L) ) \tr(L) + \tr(r) + (n-1) \frac{\varepsilon}{2}.
\end{align*}

Combining this with the conservation law \eqref{GeneralConservationLaw}, it follows that on any $C^3$-regular cohomogeneity one Ricci soliton there holds
\begin{equation}
2 \ddot{u} = (d_S + 1) \ddot{u}(0) + \varepsilon u + \dot{u}^2 + \tr(L^2) - (\tr(L))^2 + \tr(r) + (n-1) \frac{\varepsilon}{2}.
\label{EquationForDDotU}
\end{equation}

In order to estimate the quantities in this equation, firstly recall that in the steady and expanding case, $\varepsilon \geq 0,$ the soliton potential is negative for all $t>0$ according to proposition \ref{PotentialFunctionOfExpandingRSAlongCohomOneFlow}. Secondly, it follows from the structure of the scalar curvature of $P_t$ in \eqref{ScalarCurvatureOfHomogeneousSpaceB} that $\tr(r_t) \leq \frac{1}{2} \sum_i \frac{d_i b_i}{f_i^2(t)}.$ Thirdly, observe that as long as the shape operator $L$ is positive definite one has $\dot{f}_i > 0$ and hence 
\begin{align*}
\tr(r_t) \leq \frac{1}{2} \sum_i \frac{d_i b_i}{f_i^2(t_{*})} = - \frac{1}{2} \tr_{g_{t_{*}}} (B_{| \mathfrak{p}})
\end{align*}
for all $t \geq t_{*}.$ Positive definiteness of $L$ also implies that $\tr(L^2) - (\tr(L))^2 \leq 0.$ Let $a>0$ and suppose that 
$(d_S+1) \ddot{u}(0) + \frac{c_{*}}{2} + (n-1) \frac{\varepsilon}{2}< - a.$ Then it follows that 
\begin{equation*}
\ddot{u}(t) \leq - a + \frac{1}{2} \dot{u}^2(t)
\end{equation*}
for all $t \geq t_{*}.$ It is useful to note that despite the singularity of \eqref{EquationForDDotU} at $t=0,$ if the assumption $-\tr_{g_{t_{*}}}(B_{| \mathfrak{p}}) < c_{*}$ already holds at $t_{*}=0,$ then the above estimate on $\ddot{u}$ also holds at $t=0.$

If $- \dot{u}(t_{*}) \geq c$ then the claim follows from the concavity of the soliton potential due to proposition \ref{PotentialFunctionOfExpandingRSAlongCohomOneFlow}. Otherwise, one can assume that $a > \frac{c^2}{2}$ and the claim then follows from the comparison principle for ODEs and remark \ref{ODEcomparisonGrowthPotentialFunction} below. $\hfill \Box$

\begin{remarkroman}
Let $a>0.$ The initial value problem 
\begin{align*}
y' = -a + \frac{y^2}{2} \ \text{ and } \ y(s_{*})=y_{*}
\end{align*}
is solved by $y(s)= \sqrt{2a} \tanh ( \sqrt{\frac{a}{2}} (s_{*} -s) + \operatorname{arctanh}( \frac{y_{*}}{\sqrt{2a}}) )$ if $-a+\frac{y_{*}^2}{2} <0.$

Fix $c>0,$ $s_0>s_{*}$ and suppose that $0 \leq -y_{*} \leq c.$ Then there exists an $a_0 \geq 0$ depending only on $c,$ $s_0-s_{*}$ such that for all $a>a_0$ the solution to the above initial value problem satisfies $-y(s) \geq c$ for all $s>s_0.$
\label{ODEcomparisonGrowthPotentialFunction}
\end{remarkroman}

This argument generalises an observation of Appleton \cite{AppletonSteadyRS} for the soliton potential of steady Ricci solitons on complex line bundles over Fano K\"ahler Einstein manifolds.

\section{Applications}
\label{ApplicationsSection}

In this section, theorem \ref{MainGrowthEstimateTheorem} will be used to construct new examples of complete steady and expanding Ricci solitons in the following geometric situations: the Dancer-Wang set-up \cite{DWCohomOneSolitons}, where previously only explicit K\"ahler Ricci solitons were known, the examples of L\"u-Page-Pope \cite{LuPagePopeQEinstein}, who described explicit Ricci flat metrics, and the two summands case \cite{WinkSolitonsFromHopfFibrations}, where now all low dimensional examples are covered. 

\subsection{Revision of the two summands case} This section revisits the construction of complete steady and expanding Ricci solitons in the two summands case and proves Theorem \ref{TheoremOnLowDimTwoSummands}. The method also recovers Ricci solitons which were previously constructed in \cite{WinkSolitonsFromHopfFibrations} with different techniques. The results can be summarised as follows:

\begin{theorem} 
On $CaP^2 \setminus \left\{ \text{ point } \right\},$ $\mathbb{H}P^{m+1} \setminus \left\{ \text{ point } \right\}$ for $m \geq 1$ and on the vector bundle associated to $(G,H,K) = (Sp(m+1), Sp(1) \times Sp(m),  U(1) \times Sp(m))$ for $m \geq 1,$ there exist a $1$-parameter family of non-homothetic complete steady and a $2$-parameter family of non-homothetic complete expanding Ricci solitons.

The steady Ricci solitons are asymptotically paraboloid and thus non-collapsed. The expanding Ricci solitons are asymptotically conical.
\label{SummaryRSInTwoSummandsCase}
\end{theorem}

The two summands case describes the setting in which the isotropy representation of $P=G/K$ decomposes into two inequivalent $\Ad_K$-invariant irreducible summands $\mathfrak{p} = \mathfrak{p}_1 \oplus \mathfrak{p}_2.$ Suppose that $b$ is a biinvariant metric on $G$ which induces the metric of constant curvature $1$ on $H/K.$ If $\mathfrak{p}_1$ and $\mathfrak{p}_2$ correspond to the tangent spaces of the collapsing sphere $H/K$ and the singular orbit $G/H,$ respectively, the metric on the principal orbit is given by 
\begin{equation*}
g_t = f_1(t)^2 b_{|\mathfrak{p}_1} + f_2(t)^2 b_{|\mathfrak{p}_2}
\end{equation*}
and the Ricci soliton equations take the form
\begin{align*}
\ddot{u} & = \sum_{i=1}^2 d_i \frac{\ddot{f}_i}{f_i} - \frac{\varepsilon}{2}, \\
\frac{d}{dt} \frac{\dot{f}_1}{f_1} & = - (- \dot{u} + \tr(L)) \frac{\dot{f}_1}{f_1} + \frac{\varepsilon}{2} + \frac{A_1}{d_1} \frac{1}{f_1^2} + \frac{A_3}{d_1} \frac{f_1^2}{f_2^4}, \\
 \frac{d}{dt} \frac{\dot{f}_2}{f_2} & = - (- \dot{u} + \tr(L)) \frac{\dot{f}_2}{f_2} + \frac{\varepsilon}{2} + \frac{A_2}{d_2} \frac{1}{f_2^2} - 2 \frac{A_3}{d_2} \frac{f_1^2}{f_2^4},
\end{align*}
where $d_i = \dim \mathfrak{p}_i$ and the constants $A_i \geq 0$ are determined by the Riemannian submersion $(G/K,b_{|\mathfrak{p}}) \to (G/H,b_{|\mathfrak{p}_2}),$ see \cite{BohmInhomEinstein}: If $|| A ||$ is the norm of the associated O'Neill tensor and $\Ric^{G/H}$ is the Einstein constant of $(G/H,b_{|\mathfrak{p}_2}),$ then one has $A_1 = d_1 (d_1-1),$ $A_2 = d_2 \Ric^{G/H}$ and $A_3 = d_2 || A ||^2.$ In the following,
\begin{align*}
A_2, \ A_3 > 0 \ \text{ and } \ A_1 = d_1(d_1-1) \geq 0
\end{align*}
will be assumed.

In order to smoothly extend the metric and the soliton potential over the singular orbit, impose the smoothness conditions
\begin{align}
& f_1(0)=0, \ \dot{f}_1(0)=1 \ \text{ and }  \
f_2(0)= \bar{f} > 0, \ \dot{f}_2(0)=0, \label{TwoSummandsRefinedInitialConitions} \\ 
& \hspace{23pt} u(0) = 0, \ \dot{u}(0)=0, \ (d_1 + 1) \ddot{u}(0) =  C \leq 0 \nonumber,
\end{align}
where $C$ is the constant appearing in the conservation law \eqref{GeneralConservationLaw}. Higher regularity of the solution then follows from elliptic regularity, cf. \cite{BuzanoInitialValueSolitons}. In particular, any solution is uniquely determined by fixing $\bar{f} > 0$ and $C \leq 0$ and if $C=0$ the trajectory is automatically Einstein. Recall that according to results of Chen \cite{ChenStrongUniquenessRF} any {\em complete} steady and expanding Ricci soliton with $u(0)=0$ necessarily satisfies $C \leq 0.$

Fix $\varepsilon \geq 0.$ Then $(d_1+1)\frac{\ddot{f}_2(0)}{f_2(0)} = \frac{\varepsilon}{2} + \frac{A_2}{d_2} \frac{1}{\bar{f}^2} >0$ implies that $\dot{f}_2(t) > 0$ holds for sufficiently small $t >0.$ A key step in the proof of completeness of the corresponding Ricci soliton metrics is to show that $\dot{f}_2(t) > 0$ holds in fact for all $t>0.$ Since $\dot{f}_1 >0$ is immediate, this also implies that the shape operator $L = \diag \left( \frac{\dot{f}_1}{f_1} \id_{d_1}, \frac{\dot{f}_2}{f_2} \id_{d_2} \right)$ is positive definite. In order to prove that $\dot{f}_2 > 0$ is preserved if $\varepsilon \geq 0$, it suffices to show that 
\begin{align*}
\omega = \frac{f_1}{f_2} < \sqrt{\frac{A_2}{2 A_3}}
\end{align*}
is preserved if $\varepsilon \geq 0.$ To this end, observe that $\omega$ satisfies
\begin{align}
\frac{d}{dt} \omega  = & \ \omega \left\lbrace \frac{\dot{f}_1}{f_1} - \frac{\dot{f}_2}{f_2} \right\rbrace, \nonumber \\
\frac{d^2}{dt^2} \omega = & \ \dot{\omega} \left\lbrace \frac{\dot{f}_1}{f_1} - \frac{\dot{f}_2}{f_2} - (-\dot{u} + \tr(L)) \right\rbrace  
+ \frac{\omega}{f_1^2} \left\lbrace \frac{A_1}{d_1} - \frac{A_2}{d_2} \omega^2 + A_3 \left( \frac{1}{d_1} + \frac{2}{d_2} \right)  \omega^4 \right\rbrace
\label{SecondDerivativeOfOmega}
\end{align}
and the initial conditions $\omega(0)=0$ and $\dot{\omega}(0)= \bar{f}^{-1}.$ Notice that the function
\begin{equation}
f( \omega ) = \frac{A_1}{d_1} - \frac{A_2}{d_2} \omega^2 + A_3 \left( \frac{1}{d_1} + \frac{2}{d_2} \right)  \omega^4
\label{ZeroOrderTermSecondDerivativeOmega}
\end{equation}
has two positive zeros $\omega_{1,2}^2 = \frac{A_2}{2A_3} \frac{d_1}{2d_1+d_2} \pm \sqrt{ D }$ if and only if the discriminant
\begin{equation}
D = \left( \frac{A_2}{2 A_3} \frac{d_1}{2d_1+d_2} \right)^2  - \frac{A_1}{A_3} \frac{d_2}{2d_1+d_2}
\label{DiskriminatConeSolutions}
\end{equation}
is positive. In this case it follows that $\omega_1^2 < \frac{A_2}{4A_3}$ and $\omega_2^2 < \frac{A_2}{2A_3}.$ In particular, $\dot{f}_2 >0$ as long as $\omega < \omega_2.$

Furthermore, the geometric condition $A_1=d_1 (d_1-1)$ implies that $\omega_1=0$ if and only if $d_1=1,$ which turns the case of a collapsing circle into a special case, see proposition \ref{CompletenessDimOneIsOneWithPotentialFunctionTrick}.

\begin{remarkroman}
According to \cite{WinkSolitonsFromHopfFibrations}, the asymptotic behaviour of complete Ricci soliton solutions with $\omega^2 < \frac{A_2}{2 A_3}$, and hence $L_t > 0,$ is linked to the existence of cone solutions, certain explicit solutions of the two summands Ricci soliton equations with conical singularities at the singular orbit. However, these cone solutions are real valued if and only if $D \geq 0$. In the steady case, proposition \ref{PosDefShapeOperatorNecessaryIfEpsZeroTwoSummands} implies that $D \geq 0$ is in fact a necessary condition for complete metrics to exists.

Inserting the definitions of the constants $A_1,$ $A_2,$ $A_3$ into \eqref{DiskriminatConeSolutions} one obtains
\begin{align*}
D \geq 0 \ \text{ if and only if } \ \frac{(\Ric^{G/H})^2}{4 ||A||^2} \geq (2d_1+d_2) \frac{d_1-1}{d_1}.
\end{align*} 
In particular, $D \geq 0$ is satisfied by all examples which are induced by the Hopf fibrations, cf. \cite{WinkSolitonsFromHopfFibrations}. 
%
%
%

Naturally, the discriminant $D$ also occurs in B\"ohm's \cite{BohmInhomEinstein, BohmNonCompactEinstein} work on Einstein metrics in the two summands case.
\label{RemarkOnDiskriminant}
\end{remarkroman}

\begin{lemma}
Suppose that $d_1>1,$ $A_1=d_1(d_1-1)$ and $D \geq 0.$ As long as $\dot{f}_2 > 0$ and $\omega < \omega_2,$ the following estimate holds
\begin{align*}
\ddot{\omega} 
\leq
\frac{d_1-1}{\omega}  \left( \frac{1}{\bar{f}^2} - \dot{\omega}^2 \right) + \dot{\omega}( \dot{u} - n \frac{\dot{f}_2}{f_2} ).
\end{align*}

In particular, if furthermore $C=0$ or $C < 0$ and $\varepsilon \geq 0,$ then $\dot{\omega} \geq \frac{1}{\bar{f}}$ implies that $\ddot{\omega} \leq 0,$ and thus the bound $\dot{\omega}(t) \leq \frac{1}{\bar{f}}$ holds for $t \geq 0$ and as long as $\omega(t) < \omega_2.$
\label{wBoundedByInitialValue}
\end{lemma}
\begin{proof}
The function $f(\omega) = \frac{2d_1+d_2}{d_1 d_2} A_3 (\omega_1^2 - \omega^2)(\omega_2^2 - \omega^2)$ of \eqref{ZeroOrderTermSecondDerivativeOmega} has a local maximum at the origin and in particular one has $f(\omega) \leq f(0) = \frac{A_1}{d_1}=d_1-1$ if $0 \leq \omega \leq \omega_2.$ Since $f_2$ is increasing there holds $f_2(t) \geq f_2(0) = \bar{f}$ for $t \geq 0.$ Then first claim follows from
\begin{equation*}
\ddot{\omega} = \frac{1}{\omega}  \left( \frac{f(\omega)}{f_2^2} - (d_1-1) \dot{\omega}^2 \right) + \dot{\omega}( \dot{u} - n \frac{\dot{f}_2}{f_2} ).
\end{equation*}
If $\varepsilon \geq 0$ and $C \leq 0,$ the soliton potential either vanishes identically or is strictly decreasing by proposition \ref{PotentialFunctionOfExpandingRSAlongCohomOneFlow}. In particular, there holds $\dot{u} \leq 0$ and this completes the proof.
\end{proof} 

Theorem \ref{SummaryRSInTwoSummandsCase} now follows from proposition \ref{CompletenessTwoSummandsWithPotentialFunctionTrick} below, which moreover  produces complete steady and expanding Ricci solitons if $D \geq 0$ and $d_1>1$ and hence covers all known geometric examples. The case of a collapsing circle $d_1=1$ will be considered in proposition \ref{CompletenessDimOneIsOneWithPotentialFunctionTrick}.

\begin{proposition}
Suppose that $d_1 > 1,$ $A_1=d_1(d_1-1)$, $D > 0$ and $\varepsilon \geq 0.$ For any $\bar{f}_0>0$ there exists a constant $C_0 \leq 0$ such that for any $C < C_0$ and any $\bar{f} > \bar{f}_0$ the steady and expanding two summands Ricci soliton metrics with initial condition \eqref{TwoSummandsRefinedInitialConitions} are complete. 

In particular, these give rise to a $1$-parameter family of complete steady and a $2$-parameter family of complete expanding Ricci solitons. 
\label{CompletenessTwoSummandsWithPotentialFunctionTrick}
\end{proposition}
\begin{proof}
Completeness follows from proposition \ref{CompletenessForPositivityOfShapeOperator} once the invariance of the set 
\begin{align}
\left\{ \ \dot{f}_2 > 0 \ \text{ and } \ 0 < \omega < \omega_2 \ \right\}
\label{InvariantSetTwoSummands}
\end{align}
under the Ricci soliton ODE has been established for all $t>0.$ 

Observe that it directly follows from the two summands Ricci soliton ODE that $\omega < \omega_2$ implies $\dot{f}_2 > 0.$ It therefore remains to show that there are initial conditions as claimed so that $\omega$ cannot exceed $\omega_2$ if $\dot{f}_2 > 0.$ 

Recall that $A_1>0$ implies that $\omega_1 >0.$ Suppose there exist $t_{*},$ $\delta > 0$ such that $\omega < \omega_2$ on $[0, t_{*}],$ $\omega \leq \omega_1$ on $(t_{*}-\delta,t_{*})$ and $\omega > \omega_1$ on $(t_{*},t_{*} + \delta).$ Set 
\begin{align*}
t^{*} = \sup \ \lbrace \ t \geq t_{*} \ | \ \omega(s) \in [\omega_1, \omega_2] \ \text{ for all } \ s \in (t_{*}, t) \ \rbrace.
\end{align*}

Then $\omega(t_{*}) = \omega_1$ and the monotonicity of $f_2$ implies that $f_1(t_{*}) \geq \omega_1 \bar{f}.$ It follows a posteriori, by increasing $t_{*}$ slightly if necessary, that $\dot{\omega}(t_{*})>0$ can be assumed.

As long as $\omega_1 \leq \omega \leq \omega_2,$ the ODE \eqref{SecondDerivativeOfOmega} implies that
\begin{align*}
\frac{d}{dt} \dot{\omega}  \leq \dot{\omega} \left\lbrace \dot{u} - (d_1-1) \frac{\dot{f}_1}{f_1}-(d_2+1)\frac{\dot{f}_2}{f_2} \right\rbrace 
\end{align*}
and with lemma \ref{wBoundedByInitialValue} it follows that
\begin{align*}
\dot{\omega}(t) \leq \dot{\omega}(t_{*}) \frac{f_1^{d_1-1}(t_{*}) f_2^{d_2+1}(t_{*})}{f_1^{d_1-1}(t) f_2^{d_2+1}(t)} \exp( u(t)-u(t_{*}) ) \leq  \dot{\omega}(t_{*}) \leq \frac{1}{\bar{f}}
\end{align*}
for all $t \in (t_{*}, t^{*}).$ Notice in particular that $\omega < \omega_2$ holds on the interval $(t_{*},t_{*}+T),$ where $T=\bar{f}_0(\omega_2-\omega_1).$

Furthermore, due to the monotonicity of $f_1,$ $f_2$ one has 
\begin{equation}
- \tr_{g_t}(B_{|\mathfrak{p}}) = \sum_{i=1}^2 \frac{A_i}{f_i^2(t)} \leq \sum_{i=1}^2 \frac{A_i}{f_i^2(t_{*})} \leq \frac{A_1}{\omega_1^2 \bar{f}^2} + \frac{A_2}{\bar{f}^2} < \infty
\label{CheckAssumptionOnInitialConditionTwoSummands}
\end{equation}
for $t \geq t_{*}$ and as long as $\omega(t) < \omega_2.$

Now, according to theorem \ref{MainGrowthEstimateTheorem}, the following holds: For any $c>0,$ any $\tau \in (0,T)$ and any $\bar{f}_0>0$ there exists a constant $C_0 \leq 0$ such that along any two summands Ricci soliton trajectory whose initial conditions \eqref{TwoSummandsRefinedInitialConitions} satisfy $C<C_0$ and $\bar{f} > \bar{f}_0$ the soliton potential obeys the estimate
\begin{equation*}
u(t) \leq u(\tau+t_{*}) - c(t-(\tau+t_{*}))
\end{equation*}
for $t \geq \tau+t_{*}$ and as long as $\omega(t) < \omega_2.$ (In fact, due to the concavity properties of $u$ in proposition \ref{PotentialFunctionOfExpandingRSAlongCohomOneFlow}, the estimate actually holds for all $t \geq \tau+t_{*}.$)

The monotonicity of $f_1, f_2$ and $u$ and the a priori bound on $\dot{\omega}$ in lemma \ref{wBoundedByInitialValue} then imply that for $t \in [\tau +t_{*}, t^{*}]$ there holds
\begin{align*}
\omega(t) &
\leq \omega(\tau+t_{*}) + \dot{\omega}(t_{*}) \int_{\tau+t_{*}}^t \exp( u(s)-u(t_{*}) ) ds \\ 
& \leq \omega(\tau+t_{*}) +  \frac{\dot{\omega}(t_{*})}{c} \\
& \leq \omega_1 + \frac{1}{\bar{f}_0} \left( \tau + \frac{1}{c} \right) < \omega_2,
\end{align*}
if $c>0$ is chosen large enough and $\tau>0$ small enough. In particular, it follows that $t^{*}=\infty$ unless $\omega(t^{*})=\omega_1.$ 

However, observe that the choices of $c,$ $\tau >0$ and hence $C_0 \leq 0$ depend on $\bar{f}_0 >0$ but not on $t_{*}.$ This implies that the corresponding trajectories with $C < C_0$ indeed satisfy $\omega(t) < \omega_2$ for all $t \geq 0,$ which completes the proof. 
\end{proof}

\begin{remarkroman} 
(a) According to \cite[Proposition 2.6]{WinkSolitonsFromHopfFibrations}, one can actually choose $C_0=0$ in proposition \ref{CompletenessTwoSummandsWithPotentialFunctionTrick} if $d_1 >1$ and $(d_1+1)A_2^2 > 4 d_1 d_2 (2d_1+d_2) A_3.$ 

(b) The computation of the asymptotic behaviour of the Ricci soliton metrics in \cite{WinkSolitonsFromHopfFibrations} carries over to the metrics of Theorem \ref{TheoremOnLowDimTwoSummands} since they satisfy $L_t > 0$ and $\omega \leq \omega_2 < \frac{A_2}{2 A_3}$. In particular the steady Ricci solitons are asymptotically paraboloid and the expanding Ricci solitons are asymptotically conical.
\end{remarkroman}

In the steady case, in fact any complete metric must necessarily have positive definite shape operator:

\begin{proposition}
Let $d_1 \geq 1,$ $A_1 \geq 0$ and $\varepsilon = 0$ and let $(f_1(t),f_2(t),u(t))_{t \in [0, T)}$ be a maximal solution to the two summands Ricci solution ODE with initial condition \eqref{TwoSummandsRefinedInitialConitions}. Suppose there exists $t_0 >0$ such that $\dot{f}_2(t_0) < 0.$ 

Then $T < \infty$ and the associated metric is incomplete.
\label{PosDefShapeOperatorNecessaryIfEpsZeroTwoSummands}
\end{proposition}
\begin{proof}
It follows from \cite[Lemma 2.4]{WinkSolitonsFromHopfFibrations} that $\dot{f}_2(t) < 0$ for all $t \geq t_0.$ This implies $\dot{\omega}(t) > 0$ for all $t \geq t_0$ and hence there exists a constant $c_1>0$ such that
\begin{align*}
\frac{d}{dt} \frac{\dot{f}_1}{f_1} \geq - ( - \dot{u} + \tr(L) ) \frac{\dot{f}_1}{f_1} + c_1.
\end{align*}
Thus, by comparison, there exists a constant $c_2>0$ such that 
\begin{align*}
\frac{\dot{f}_1}{f_1} \geq c_1 \frac{e^u}{f_1^{d_1} f_2^{d_2}} + c_2,
\end{align*}
which implies 
\begin{align*}
\frac{d}{dt} \tr(L) = \ddot{u} - \tr(L^2) \leq - c_2^2 < 0
\end{align*}
due to proposition \ref{PotentialFunctionOfExpandingRSAlongCohomOneFlow}. 

However, if $T = \infty$ this contradicts the estimate $\tr(L) \to 0$ as $t \to \infty$ which holds on any complete steady cohomogeneity one Ricci soliton with a singular orbit according to proposition \ref{AsymptoticsOfPotentialFunction}.
\end{proof}

Notice that condition \eqref{CheckAssumptionOnInitialConditionTwoSummands} cannot hold at $t=0$ if $A_1>1$ because $f_1(0)=0$ implies that $\tr_{g_t}(B_{|\mathfrak{p}})$ diverges as $t \to 0.$ 

However, if $A_1 = 0$ one can actually choose $t_{*}=0$ and \eqref{CheckAssumptionOnInitialConditionTwoSummands} holds initially. Since $A_1 = d_1 (d_1-1)$ holds in geometric applications, the proof can hence be simplified if $d_1 =1.$ In this case, the two summands system can be used to construct Ricci soliton metrics on complex line bundles over Fano K\"ahler Einstein manifold, see \cite{WinkSolitonsFromHopfFibrations} for the set-up and examples of {\em non}-K\"ahler steady and expanding Ricci solitons. The steady Ricci solitons have independently been constructed by  Appleton \cite{AppletonSteadyRS} and Stolarski  \cite{StolarskiSteadyRSOnCxLineBundles}. Explicit {\em K\"ahler} Ricci solitons have been found earlier by Cao \cite{CaoSoliton} and Feldman-Ilmanen-Knopf \cite{FIKSolitons}.

In the case of steady Ricci solitons, the following argument is due to Appleton \cite{AppletonSteadyRS}, who moreover exhibits a non-collapsed Ricci soliton in this family. Stolarski \cite{StolarskiSteadyRSOnCxLineBundles} independently describes the steady Ricci solitons. Proposition \ref{CompletenessDimOneIsOneWithPotentialFunctionTrick} extends their results to the expanding case.

\begin{proposition}
Let $d_1=1,$ $A_1 =0$ and $\varepsilon \geq 0.$ For any $\bar{f}_0>0$ there exists a constant $C_0 \leq 0$ such that for all $C < C_0$ and $\bar{f}>\bar{f}_0$ the steady and expanding Ricci soliton trajectories with initial condition \eqref{TwoSummandsRefinedInitialConitions} give rise to families of complete Ricci soliton metrics on complex line bundles over Fano K\"ahler Einstein manifolds.

In particular, these families include examples of non-K\"ahler type.
\label{CompletenessDimOneIsOneWithPotentialFunctionTrick}
\end{proposition}
\begin{proof}
Recall that $A_1=0$ implies that $\omega_1=0$ and $\omega_2 = \frac{A_2}{A_3} \frac{d_1}{2d_1 + d_2}.$ As in the proof of proposition \ref{CompletenessTwoSummandsWithPotentialFunctionTrick}, it suffices to show that \eqref{InvariantSetTwoSummands} is preserved.

Clearly, $\omega < \omega_2$ still implies $\dot{f}_2 > 0$ and as $\omega(0)=0$ the condition $\omega < \omega_2$ holds initially. Moreover, since $d_1=1$ and $u(0)=0,$ one concludes as in the proof of proposition \ref{CompletenessTwoSummandsWithPotentialFunctionTrick} that
\begin{align*}
\dot{\omega}(t) \leq \dot{\omega}(0) \frac{f_2^{d_2+1}(0)}{f_2^{d_2+1}(t)} \exp( u(t) ) \leq \dot{\omega}(0) = \frac{1}{\bar{f}}
\end{align*}
as long as $\omega(t) < \omega_2.$ In particular, this condition is satisfied on $[0, \omega_2 \bar{f}).$ 

Furthermore, $- \tr_{g_t}(B_{| \mathfrak{p}}) = \sum_{i=1}^2 \frac{A_i}{f_i^2(t)} = \frac{A_2}{f_2^2(t)} \leq \frac{A_2}{\bar{f}^2} < \infty$ also holds as long as $\omega(t) < \omega_2.$ Thus, theorem \ref{MainGrowthEstimateTheorem} implies the following:

Let $\bar{f}_0 >0,$ $c>0$ and $\tau \in (0, \omega_2 \bar{f}_0).$ Then there exists a constant $C_0 \leq 0$ such that along any two summands Ricci soliton trajectory with $d_1=1,$ $A_1=0$ and $C<C_0,$ $\bar{f}>\bar{f}_0$ in \eqref{TwoSummandsRefinedInitialConitions}, the soliton potential satisfies $u(t) \leq u(\tau) - c(t-\tau)$ for $t \geq \tau$ and as long as $\omega(t) < \omega_2.$ 

Thus, one concludes that
\begin{align*}
\omega(t) 
 \leq \omega(\tau) + \dot{\omega}(0) \int_{\tau}^t \exp(u(s)) ds 
 \leq \frac{1}{\bar{f}} \left( \tau + \frac{1}{c} \right) < \omega_2
\end{align*}
for all $t \geq \tau$ if $c>0$ is chosen large enough and $\tau>0$ sufficiently small. In particular, in this case, $\omega(t) < \omega_2$ holds for all $t \geq 0.$ 
\end{proof}

\begin{remarkroman}
It follows from \cite[Proposition 2.13]{WinkSolitonsFromHopfFibrations} that proposition \ref{CompletenessDimOneIsOneWithPotentialFunctionTrick} holds with $C_0 = 0$ provided $A_2^2 > 2d_2(d_2+2) A_3.$ 
\end{remarkroman}

\subsection{Generalisations of the circle bundle case}
\label{CircleBundleCaseWithGrowthOfPotentialFunction}
In this section two generalisations of proposition \ref{CompletenessDimOneIsOneWithPotentialFunctionTrick} will be discussed.

\subsubsection{The Dancer-Wang set-up}
\label{DancerWangSetUpSection}
Extending previous work of Cao \cite{CaoSoliton} and Feldman-Ilmanen-Knopf \cite{FIKSolitons}, Dancer-Wang \cite{DWCohomOneSolitons} found explicit examples of K\"ahler Ricci solitons on complex line bundles over {\em products} of Fano K\"ahler Einstein manifolds. Their construction assumes that the Euler class of the line bundle is a rational linear combination of the first Chern classes of the base manifolds. However, the K\"ahler condition imposes strong restrictions on the coefficients. It will be shown that in fact on any of these line bundles there exist continuous families of steady and expanding Ricci solitons, thus proving Theorem \ref{RSOnLineBundlesOverProductsOfFanos}. The precise geometric set-up is the following:

Let $\left\lbrace (V_i,J_i, h_i) \right\rbrace _{i=1, \ldots, m}$ be Fano K\"ahler Einstein manifolds of real dimension $d_i.$ Then the first Chern classes are $c_1(V_i, J_i)= p_i \rho_i$ for positive integers $p_i$ and indivisible classes $\rho_i \in H^2(V_i, \mathbb{Z}).$ Suppose that the K\"ahler metrics $h_i$ are normalised such that $\Ric(h_i) = p_i h_i$ and let $\eta_i$ be the associated K\"ahler form. Choose $q_1, \ldots, q_m \in \Z \setminus \lbrace 0 \rbrace$ and consider the principal circle bundle $\pi \colon P \to B$ with Euler class $\sum_{i=1}^m q_{i} \pi_i^{*} \rho_i.$ Notice that on $P$ there exists a principal connection $\theta$ with curvature $\Omega = \sum_{i=1}^m q_{i} \pi_i^{*} \eta_i.$ 

Then the Ricci soliton equations on $I \times P$ with respect to the Riemannian metric $g = dt^2 + g_t,$ where $g_t = f^2(t) \theta \otimes \theta + \sum_{i=1}^m g_{i}^2(t) \pi_{i}^{*} h_{i},$ are given by
\begin{align}
\ddot{u} & \ =  \frac{\ddot{f}}{f} + \sum_{i=1}^m d_i \frac{\ddot{g}_i}{g_i} - \frac{\varepsilon}{2}, \label{CircleBundleRSequationI} \\
\frac{d}{dt} \frac{\dot{f}}{f} & \ = -  (- \dot{u} + \tr(L)) \frac{\dot{f}}{f} + \frac{\varepsilon}{2}  + \sum_{i=1}^m \frac{d_i q_i^2}{4} \frac{f^2}{g_i^4}, \label{CircleBundleRSequationII} \\
\frac{d}{dt} \frac{\dot{g}_i}{g_i} & \ =  - (-\dot{u} + \tr(L) ) \frac{\dot{g}_i}{g_i}   +  \frac{\varepsilon}{2} + \frac{p_i}{g_i^2} - \frac{q_i^2}{2} \frac{f^2}{g_i^4}, \label{CircleBundleRSequationIII}
\end{align}
for $i=1, \ldots, m,$ and where $L = \diag \left( \frac{\dot{f}}{f}, \frac{\dot{g}_1}{g_1} \mathbb{I}_{d_1}, \ldots, \frac{\dot{g}_m}{g_m} \mathbb{I}_{d_m} \right)$ 
is the shape operator of the principal hypersurface. A solution to these equations on $[0, \infty) \times P$ induces a complete Ricci soliton metric on the associated complex line bundle if the smoothness conditions
\begin{align}
f(0)=0, & \ \dot{f}(0)=1 \ \text{ and } \ g_i(0) >0, \ \dot{g}_{i}(0) = 0, \label{DancerWangInitialConditions} \\
 & \hspace{-0.6mm} u(0)=\dot{u}(0)=0, \ 2 \ddot{u}(0) = C \nonumber
\end{align}
are imposed, where $C$ is the constant in the conservation law \eqref{GeneralConservationLaw}. Higher regularity again follows from the work of Buzano \cite{BuzanoInitialValueSolitons}. Recall that if $\varepsilon \geq 0$ then $C \leq 0$ is a necessary condition to obtain a complete solution.

\vspace{2mm}

The metric on the associated complex line bundle is K\"ahler if $\frac{d}{dt} g_i^2 = - q_i f$ for $i=1, \ldots, m,$ which necessarily requires $q_i<0,$ cf. \cite{WWEinsteinS2Bundles}. 

\begin{remarkroman} For $\varepsilon \geq 0,$ notice that  \eqref{CircleBundleRSequationII} and \eqref{DancerWangInitialConditions} imply that $\frac{\dot{f}(t)}{f(t)} > 0$ for $t>0$ and as long as the solution exists. Proposition \ref{PotentialFunctionOfExpandingRSAlongCohomOneFlow} hence implies that any estimate $- \dot{u} \geq c$ is preserved along solutions of the Ricci soliton ODE with $\varepsilon \geq 0.$
\label{PotentialDecreasingInCircleBundleCase}
\end{remarkroman}

Fix $\varepsilon \geq 0$ and observe that $2 \frac{\ddot{g}_i(0)}{g_i(0)} = \frac{\varepsilon}{2} + \frac{p_i}{g_i^2(0)} >0$ implies that $\dot{g}_i(t)>0$ for sufficiently small $t>0.$

Suppose that $C_1 \geq 1$ is a constant to be specified later and that $Q_{ij} = \frac{g_i}{g_j}$ and $\omega_i = \frac{f}{g_i}$ are bounded by
\begin{equation}
Q_{ij} \leq C_1 \ \text{ and } \ \omega_i^2 \leq C_{2,i}^2 = \frac{1}{m C_1^2} \frac{4}{(d_i+2)q_i^2} \min_{k \in \lbrace1, \ldots, m\rbrace} \lbrace p_k \rbrace
\label{APrioriBounds}
\end{equation}
for all $i,j = 1, \ldots, m.$ Define
\begin{equation}
C_2 = \min_{i \in \lbrace 1, \ldots, m \rbrace} \lbrace C_{2,i} \rbrace.
\label{UniformAPrioriBoundOmega}
\end{equation} 

Notice that $\omega_i \leq C_2$ is automatically satisfied initially as $\omega_i(0)=0$ and that condition \eqref{APrioriBounds} already implies
\begin{align}
\frac{p_i}{g_i^2} - \frac{q_i^2}{2} \frac{f^2}{g_i^4} \geq \frac{d_i p_i}{d_i+2} \frac{1}{g_i^2}.
\label{KeyEstimateOnRHSinODEforGI}
\end{align}
In particular, as long as \eqref{KeyEstimateOnRHSinODEforGI} holds one has $\dot{g}_i > 0$ for $i=1, \ldots, m.$ The condition on $Q_{ij}$ can be satisfied initially by simply choosing $C_1 \geq 1$ sufficiently large and $g_i(0)$ accordingly.

\begin{remarkroman}
Notice that the a priori bound \eqref{APrioriBounds} also implies that all $Q_{ij}$ are bounded from below. 

If $m=1$ one can choose $C_1=1$ and the condition $\omega^2 < \frac{4p}{(d+2)q^2}$ is the same as in proposition  \ref{CompletenessDimOneIsOneWithPotentialFunctionTrick}.
\label{OptimalOmegaBoundIfMequalsOne}
\end{remarkroman}

It will be shown that for a suitable choice of initial conditions in \eqref{DancerWangInitialConditions} the bounds in \eqref{APrioriBounds} are preserved along the corresponding steady and expanding Ricci soliton trajectories.

Observe that $\omega_i= \frac{f}{g_i}$ satisfies the equations
\begin{align*}
\dot{\omega}_i = & \ \omega_i \left\lbrace \frac{\dot{f}}{f} - \frac{\dot{g}_i}{g_i} \right\rbrace \ \text{ and } \\
\ddot{\omega}_i =  
& \ \dot{\omega}_i \left\lbrace \dot{u} - (d_i+1) \frac{\dot{g}_i}{g_i} - \sum_{j \neq i} d_j \frac{\dot{g}_j}{g_j}  \right\rbrace 
+ \frac{\omega_i}{g_i^2} \left\lbrace \frac{q_i^2}{2} \omega_i^2 - p_i +  \sum_{j=1}^m \frac{d_j q_j^2}{4}  \frac{f^2}{g_j^2} \frac{g_i^2}{g_j^2} \right\rbrace.
\end{align*}

\begin{remarkroman}
Any critical point of $\omega_i$ along a solution which satisfies \eqref{APrioriBounds} is a local maximum. Indeed, one has 
\begin{align*}
\frac{d_i+2}{4} q_i^2 \omega_i^2 - p_i + \sum_{j \neq i} \frac{d_j q_j^2}{4} \frac{f^2}{g_j^2} \frac{g_i^2}{g_j^2} 
& \ \leq \frac{d_i+2}{4} q_i^2 \omega_i^2 - p_i + \frac{m-1}{m} \min \left\lbrace p_k \right\rbrace \\
& \ \leq \left( \frac{1}{m C_1^2} + \frac{m-1}{m} \right) \min \left\lbrace p_k \right\rbrace - p_i \leq 0
\end{align*}
and equality implies $m=1,$ $C_1=1$ and equality in $\eqref{APrioriBounds},$ in which case $\omega_1$ is maximal by assumption. Therefore $\dot{\omega}_i$ can change its sign at most once. In particular, if the solution is defined for all times $t \geq 0,$ then $\omega_i$ converges as $t \to \infty.$
\label{ConvergenceOmegaI}
\end{remarkroman}

\begin{lemma} Let $\varepsilon \geq 0,$ $C_1 \geq 1$ and fix $\bar{g}>0,$ $\varepsilon_0 \in (0, C_2],$ where $C_2$ is defined as in \eqref{UniformAPrioriBoundOmega}. Then there exists a constant $C_0 \leq 0$ such that along any Ricci soliton trajectory with $C < C_0$ and $g_i(0) > \bar{g}$ for all $i=1, \ldots, m$ in \eqref{DancerWangInitialConditions} the estimate $\omega_i < \varepsilon_0$ is preserved for all $i=1, \ldots, m$ as long as $Q_{ij}(t) \leq C_1.$
\label{OmegaSmall}
\end{lemma}
\begin{proof}
Observe that as long as \eqref{APrioriBounds} is satisfied the second term in the formula for $\ddot{\omega}_i$ is nonpositive. Since $\dot{\omega}_i(0) = \frac{1}{g_i(0)},$ this implies as before that 
\begin{align*}
\dot{\omega}_i(t) \leq \frac{g_{i}^{d_i}(0) \prod_{j \neq i} g_{j}^{d_j}(0)}{g_{i}^{d_i+1}(t) \prod_{j \neq i} g_{j}^{d_j}(t)} \exp(u(t)) \leq \frac{1}{g_i(0)} < \frac{1}{\bar{g}}
\end{align*}
for $t \geq 0$ and as long as \eqref{APrioriBounds} holds.  Observe that the bound $\omega_i \leq C_2$ clearly holds on an interval $[0,T],$ where $T>0$ only depends on $\bar{g}>0$ and $C_2>0.$

In order to apply theorem \ref{MainGrowthEstimateTheorem}, note that the Ricci soliton ODE and \eqref{KeyEstimateOnRHSinODEforGI} imply positive definiteness of the shape operator, and that 
\begin{align*}
- \tr_{g_t} (B_{|\mathfrak{p}}) = \sum_{i=1}^m \frac{d_i p_i}{g_i^2(t)} \leq \sum_{i=1}^m \frac{d_i p_i}{g_i^2(0)} \leq \overline{C}(d_i, p_i, \bar{g}) < \infty
\end{align*}
is uniformly bounded in terms of $\bar{g}>0$ and geometric or topological data. Hence, due to theorem \ref{MainGrowthEstimateTheorem} and remark \ref{PotentialDecreasingInCircleBundleCase}, for all $c > 0$ and all $\tau \in (0,T)$ there exists a constant $C_0 \leq 0$ such that for all $C < C_0$ in \eqref{DancerWangInitialConditions} one has $- \dot{u}(t) \geq c$ for all $t>\tau$. This implies 
\begin{align*}
\omega_i(t) \leq \omega_i(\tau)+ \dot{\omega}_i(0) \int_{ \tau }^t \exp(u(s))ds \leq \frac{1}{g_i(0)} \left( \tau + \frac{1}{c} \right)
\end{align*}
and hence for sufficiently small $\tau>0$ and sufficiently large $c > 0,$ the condition $\omega_i < \varepsilon_0$ is preserved along the Ricci soliton trajectory with $Q_{ij}(t) \leq C_1.$
\end{proof}

Now the existence of an a priori bound for the quotient $Q_{ij} = \frac{g_i}{g_j}$ will be established. Observe that 
\begin{align*}
\dot{Q}_{ij} = & \ Q_{ij} \left\lbrace \frac{\dot{g}_i}{g_i} - \frac{\dot{g}_j}{g_j} \right\rbrace \ \text{ and } \\
\ddot{Q}_{ij} = 
& \ \dot{Q}_{ij} \left\lbrace \dot{u} - \tr(L) + \frac{\dot{g}_i}{g_i} - \frac{\dot{g}_j}{g_j} \right\rbrace + Q_{ij}  \left\lbrace \frac{p_i}{g_i^2} - \frac{q_i^2}{2} \frac{f^2}{g_i^4} - \left( \frac{p_j}{g_j^2} - \frac{q_j^2}{2} \frac{f^2}{g_j^4} \right) \right\rbrace.
\end{align*}

It is straightforward to check that the initial conditions 
\begin{align*}
\dot{Q}_{ij}(0)=0 \ \text{ and } \ \ddot{Q}_{ij}(0) = \frac{3}{2} Q_{ij}(0) \left\lbrace \frac{p_i}{g_i^2(0)} - \frac{p_j}{g_j^2(0)} \right\rbrace
\end{align*}
hold. 


As a first step, an a priori estimate on $\dot{Q}_{ij}$ will be derived:
\begin{lemma}
As long as \eqref{KeyEstimateOnRHSinODEforGI} holds, one has the estimate 
\begin{align*}
\ddot{Q}_{ij} \leq & \ \dot{Q}_{ij} \left\lbrace \dot{u} - \tr(L) + \frac{\dot{g}_i}{g_i} - \frac{\dot{g}_j}{g_j} \right\rbrace + Q_{ij} \frac{p_i}{g_i^2} \\
= & \ \dot{Q}_{ij} \left\lbrace \dot{u} - \sum_{k \neq i,j} d_k \frac{\dot{g}_k}{g_k} - (d_i +d_j) \frac{\dot{g}_j}{g_j} \right\rbrace + \frac{1}{Q_{ij}} \left\lbrace \frac{p_i}{g_j^2} - (d_i-1) \dot{Q}_{ij}^2 \right\rbrace.
\end{align*}

In particular, if furthermore $C=0$ or $C < 0$ and $\varepsilon \geq 0,$ then $\dot{Q}_{ij}(t_0) > \sqrt{ \frac{p_i}{d_i-1} \frac{1}{g_j^2(0)} }$ implies $\ddot{Q}_{ij}(t_0) < 0$ and hence the a priori bound 
\begin{align*}
\dot{Q}_{ij}(t) \leq \max \left\lbrace \sqrt{ \frac{p_i}{d_i-1} \frac{1}{g_j^2(0)} } \right\rbrace.
\end{align*}
\label{APrioriBoundOnDotQij}
\end{lemma}
\begin{proof}
The first part is a straightforward calculation. 
%
For the second part recall that \eqref{KeyEstimateOnRHSinODEforGI} implies that $\dot{g}_i>0$ for $i=1, \ldots, m$ and that for steady and expanding Ricci solitons the soliton potential is decreasing according to proposition \ref{PotentialFunctionOfExpandingRSAlongCohomOneFlow}.
\end{proof}

\begin{lemma}
Fix $\varepsilon \geq 0$ and $\overline{G}>0.$ Set $C_1 = \max \left\lbrace \overline{G}, \sqrt{\frac{d_j+2}{d_j} \frac{p_i}{p_j}} \right\rbrace  + 1$ and define $C_2>0$ as in \eqref{UniformAPrioriBoundOmega}. Then for every $\bar{g}>0$ there exists a constant $C_0 \leq 0$ such that along any Ricci soliton trajectory with $Q_{ij}(0) \leq \overline{G}$, $g_i(0) > \bar{g}$ for $i,j=1, \ldots, m$ and $C < C_0$ in \eqref{DancerWangInitialConditions} the a priori bound \eqref{APrioriBounds} is preserved, and in fact $\omega_i \leq C_2$ for $i=1, \ldots, m.$
\label{APrioriBoundDWSetUpPreserved}
\end{lemma}
\begin{proof}
By assumption, \eqref{APrioriBounds} is satisfied initially. Choose $\delta > 0$ small. Lemma \ref{OmegaSmall} implies that there are initial conditions so that $\omega_i \leq C_2- \delta$ is preserved as long as $Q_{ij} \leq C_1$. Therefore it suffices to show that for a suitable subset of these initial conditions the estimate $Q_{ij} \leq C_1 -\delta$ is preserved as long as $\omega_i \leq C_2.$

Recall how the initial conditions in lemma \ref{OmegaSmall} are chosen: Consider $\varepsilon_0 = C_2- \delta$ in lemma \ref{OmegaSmall}. As the proof shows, for any $c>0$ and any small $\tau_0 > 0$ there exists a constant $C{'} \leq 0$ such that along any Ricci soliton trajectory with $C < C{'}$ and $g_i(0) > \bar{g}$ for all $i=1, \ldots, m$ in \eqref{DancerWangInitialConditions} one has $- \dot{u}(t) \geq c$ for all  $t > \tau_0$. Moreover, this implies that  for a sufficiently small choice of $\tau_0 > 0$ and a sufficiently large choice of $c>0$ the estimate $\omega_i \leq C_2 - \delta$ is preserved as long as $Q_{ij}(t) \leq C_1.$

In a next step, it will be shown that for a suitable choice of initial conditions, the estimate $Q_{ij} \leq C_1 - \delta$ cannot be violated as long as $\omega_i \leq C_2$ for $i=1, \ldots, m.$ Observe that, as long as \eqref{APrioriBounds} is satisfied, the equation for $\ddot{Q}_{ij}$ and \eqref{KeyEstimateOnRHSinODEforGI} imply that 
\begin{align*}
\ddot{Q}_{ij} \leq \dot{Q}_{ij} \left\lbrace \dot{u} - \tr(L) + \frac{\dot{g}_i}{g_i} - \frac{\dot{g}_j}{g_j} \right\rbrace + \frac{Q_{ij}}{g_i^2} \left\lbrace p_i - \frac{d_j p_j}{d_j+2} Q_{ij}^2 \right\rbrace.
\end{align*}
%
Now suppose that $Q_{ij}(t) \geq \sqrt{\frac{d_j+2}{d_j} \frac{p_i}{p_j}}$ for $t>t_0$ and $t_0 \geq 0$ is minimal with that property. It follows a posteriori that, by increasing $t_0$ slightly if necessary, $\dot{Q}_{ij}(t_0) > 0$ can be assumed and then 
\begin{align*}
\dot{Q}_{ij}(t) \leq \dot{Q}_{ij}(t_0) \frac{\left(g_i^{d_i-1}g_j^{d_j+1} \prod_{k \neq i,j}g_k^{d_k}\right)(t_0)}{\left(g_i^{d_i-1}g_j^{d_j+1} \prod_{k \neq i,j}g_k^{d_k}\right)(t)} \exp( u(t) - u(t_0) ) \leq \sqrt{\frac{p_i}{d_i-1} \frac{1}{g_j^2(0)}}
\end{align*}
holds for $t \geq t_0.$ Notice in particular that the bound $Q_{ij} \leq C_1 $ is still satisfied on $[t_0,t_0+T),$ where one can choose
$T = \bar{g} \cdot \max \lbrace \frac{d_i-1}{p_i} \rbrace^{-1/2}>0$ due to the definition of $C_1.$

Recall that according to \eqref{APrioriBounds} the shape operator is positive definite and that the bound $- \tr_{g_t} (B_{|\mathfrak{p}}) \leq \sum_{i=1}^m \frac{d_i p_i}{g_i^2(0)}$ was established in the proof of lemma \ref{OmegaSmall} for $t \geq 0.$

Fix $\tau_0 < T$ and apply theorem \ref{MainGrowthEstimateTheorem} with $c_{*}=\sum_{i=1}^m \frac{d_i p_i}{\bar{g}^2},$ $c>0,$ $t_{*}=t_0$ and $\tau=\tau_0+t_0.$ Due to remark \ref{PotentialDecreasingInCircleBundleCase} this yields a constant $C{''} \leq 0$ such that along any Ricci soliton trajectory with $C < C{''}$ and $g_i(0) > \bar{g}$ for all $i=1, \ldots, m$ in \eqref{DancerWangInitialConditions} one has $- \dot{u}(t) \geq c$ for all $t > \tau$. It follows that
\begin{align*}
Q_{ij}(t) & \leq Q_{ij}(\tau) + \dot{Q}_{ij}(t_0) \int_{\tau}^t \exp( u(s)-u(t_0) ) ds \\
& \leq Q_{ij}(\tau) + \frac{\dot{Q}_{ij}(t_0)}{c} \\
& \leq Q_{ij}(t_0) +  \sqrt{\frac{p_i}{d_i-1} \frac{1}{g_j^2(0)}} \left( \tau-t_0 + \frac{1}{c} \right) \leq C_1 - \delta
\end{align*}
for a sufficiently small choice of $\tau-t_0=\tau_0>0$ and a sufficiently large choice of $c>0.$

It remains to show that there is a choice of initial conditions, in particular of $\tau_0>0$ and $c>0$, that similarly works in the proof of lemma \ref{OmegaSmall}. Recall that the constants $C{'}, C{''} \leq 0$ in theorem \ref{MainGrowthEstimateTheorem}, for which the linear growth rate $- \dot{u} \geq c$ is guaranteed if $\ddot{u}(0) < C{'}, C{''},$ respectively, only depend on $c_{*}, c>0$ and the difference $\tau - t_{*}=\tau_0>0,$ but neither on $t_{*}$ nor $\tau.$ Hence there is a uniform choice of $\tau_0 \in (0,T)$ and $c>0$ which implies that \eqref{APrioriBounds} is preserved along the Ricci soliton trajectory provided that $\ddot{u}(0)< C \leq C_0 = \min \lbrace C{'}, C{''} \rbrace \leq 0$ and $g_i(0) > \bar{g}$ for all $i=1, \ldots, m.$
\end{proof}

Completeness of the Ricci soliton metrics now follows from lemma \ref{APrioriBoundDWSetUpPreserved} and proposition \ref{CompletenessDancerWangSetUp} below.

\begin{proposition} Let $\varepsilon \geq 0$ and consider a maximal Einstein or Ricci soliton trajectory in the Dancer-Wang set-up with initial conditions \eqref{DancerWangInitialConditions} and $\ddot{u}(0) \leq 0.$ Suppose that $\frac{f^2}{g_i^2}< \frac{2 p_i}{q_i^2}$ holds along the entire trajectory if $q_i \neq 0.$ 

Then this trajectory is defined for all $t \geq 0$ and corresponds to a complete Einstein or Ricci soliton metric.
\label{CompletenessDancerWangSetUp}
\end{proposition}
\begin{proof}
It is immediate that the shape operator remains positive definite along the trajectories and hence proposition \ref{CompletenessForPositivityOfShapeOperator} yields the claim.
\end{proof}

\begin{remarkroman}
In the situation of proposition \ref{CompletenessDancerWangSetUp}, the associated Einstein and Ricci soliton loci of remark \ref{PreservedRSandEinsteinLoci} define bounded regions in an associated phase space: Consider the variables
\begin{align*}
\mathcal{L} = \frac{1}{-\dot{u}+\tr(L)} \ \text{ and } \ X_0  = \mathcal{L} \cdot \frac{\dot{f}}{f}, \ X_i  = \mathcal{L} \cdot \frac{\dot{g}_i}{g_i} \ \text{ and } \
Y_0 =  \frac{\mathcal{L}}{f}, \ Y_i = \frac{\mathcal{L}}{g_i} \ \text{ and } \ \omega_i = \frac{Y_i}{Y_0}
\end{align*}
for $i =1, \ldots, m.$ Since $\omega_i^2 < \frac{2 p_i}{q_i^2}$, the conservation law \eqref{ReformulatedGeneralConsLaw} implies
\begin{align*}
\sum_{i=0}^m d_i X_i^2 + \sum_{i=1}^m \frac{d_i p_i}{2} Y_i^2 + (n-1) \frac{\varepsilon}{2} \mathcal{L}^2 \leq 1.
\end{align*}
Thus the variables $X_i, Y_i, \omega_i$ are indeed all bounded, and $\mathcal{L}$ is also bounded if $\varepsilon >0.$ This can also be used to give an alternative proof of proposition \ref{CompletenessDancerWangSetUp}: With the rescaling $\frac{d}{ds} = \mathcal{L} \cdot \frac{d}{dt}$ the Ricci soliton ODE takes the form
\begin{align*}
\frac{d}{ds} \mathcal{L} = & \ \mathcal{L} \left( \sum_{j=0}^m d_j X_j^2 - \frac{\varepsilon}{2} \mathcal{L}^2  \right), \\\frac{d}{ds} X_0 = & \ X_0 \left( \sum_{j=0}^m d_j X_j^2 - \frac{\varepsilon}{2} \mathcal{L}^2 -1 \right) + \frac{\varepsilon}{2} \mathcal{L}^2 + \sum_{j=1}^m \frac{d_j q_j^2}{4} \omega_j^2 Y_j^2, \\
\frac{d}{ds} X_i = & \ X_i \left( \sum_{j=0}^m d_j X_j^2 - \frac{\varepsilon}{2} \mathcal{L}^2 -1 \right) +\frac{\varepsilon}{2} \mathcal{L}^2 + p_i Y_i^2 - \frac{q_i^2}{2} \omega_i^2 Y_i^2 \hspace{4mm} \text{ for } \ i=1, \ldots, m, \\
\frac{d}{ds} Y_i = & \ Y_i \left( \sum_{j=0}^m d_j X_j^2 - \frac{\varepsilon}{2} \mathcal{L}^2  - X_i \right)
\hspace{44.3mm}  \text{ for } \ i=0, \ldots, m, \\
\frac{d}{ds} \omega_i = & \ \omega_i \left( X_0 - X_i \right)
\hspace{69.7mm}  \text{ for } \ i=1, \ldots, m.
\end{align*}
Since the variables remain bounded along the flow and the rescaled ODE system has a polynomial right hand side, the corresponding maximal solution is defined for all times $s.$ Similarly, the ODE for $\mathcal{L}$ shows that $\mathcal{L}$ cannot blow up in finite time. The original time slice, the soliton potential and the metric can be recovered via 
\begin{align*}
t(s) = t(s_0) + \int_{s_0}^s \mathcal{L} (\tau) d \tau \ \text{ and } \ \dot{u} = \frac{\sum_{j=0}^m d_j X_j- 1}{\mathcal{L}} \ \text{ and } \ f = \frac{\mathcal{L}}{Y_0} \ \text{ and } \ g_{i}= \frac{\mathcal{L}}{Y_i}.
\end{align*}

To show completeness of the metric, it remains to prove that $t \to \infty$ as $s \to \infty.$ Notice that the variables $\mathcal{L}, X_i, Y_i, \omega_i$ are by definition all positive initially, and the assumption $\omega_i^2 < \frac{2 p_i}{q_i}$ implies that positivity is preserved along the Ricci soliton ODE.  The arguments in \cite[Proposition 2.8 and Corollary 2.11]{WinkSolitonsFromHopfFibrations} now show that $t \to \infty$ as $s \to \infty.$

Notice that in fact $\mathcal{L} \to \frac{1}{\sqrt{-C}}$ as $s \to \infty$ due to proposition \ref{AsymptoticsOfPotentialFunction}. Hence $\mathcal{L}$ is also bounded in the case of non-trivial steady Ricci solitons.
\label{CompletenessSOneBundleViaBoundedPhaseSpaceRegion}
\end{remarkroman}

In the case of steady Ricci solitons, positive definiteness of the shape operator is actually a necessary condition for completeness.

\begin{proposition}
Let $\varepsilon = 0$ and suppose that $(f(t),g_i(t),u(t))_{t \in [0,T)}$ is a maximal solution to the Ricci soliton equations \eqref{CircleBundleRSequationI} - \eqref{CircleBundleRSequationIII} which satisfies the initial conditions \eqref{DancerWangInitialConditions} with $\ddot{u}(0) \leq 0.$ Suppose that $T = \infty.$ 

Then there holds $\dot{f}(t),$ $\dot{g}_i(t) > 0$ for all $t > 0.$ 
\label{PosDefShapeOperatorNecessaryIfEpsZeroCircleBundle}
\end{proposition}
\begin{proof}
It is clear that $\dot{f} > 0$ is preserved. Suppose that there exists $i_{0} \in \left\lbrace 1, \ldots, m \right\rbrace $ and $t_0 > 0$ such that $g_{i_0}(t_0) < 0.$ An argument analogous to the proof of \cite[Lemma 2.4]{WinkSolitonsFromHopfFibrations} shows that $g_{i_0}(t) < 0$ for all $t \geq t_0. $ As in the proof of proposition \ref{PosDefShapeOperatorNecessaryIfEpsZeroTwoSummands}, a comparison argument shows that $\frac{\dot{f}}{f}$ is bounded away from zero, which implies that the mean curvature decreases with definite speed. On the other hand, this is impossible as the mean curvature must tend to zero as $t \to \infty$ according to proposition \ref{AsymptoticsOfPotentialFunction}.
\end{proof}

The asymptotic behaviour of the steady Ricci solitons is given as follows:

\begin{proposition}
The steady Ricci soliton metrics with \eqref{APrioriBounds} are complete and if furthermore $\sum_{i=1}^m d_i +2 < m \left( \min \left\lbrace d_j \right\rbrace + 2 \right) C_1^2$ it follows that
\begin{align*}
- \dot{u}(t) \to \sqrt{-C} \ \text{ and } \ \frac{g_i^2(t)}{t} \to \frac{2p_i}{\sqrt{-C}} \ \text{ and } \ f(t) \text{ converges }
\end{align*}
as $t \to \infty.$
\label{AsymptoticsSteadyRSinProductsOverFanos}
\end{proposition}
\begin{proof}
In the steady case $f$ and $g_i$ satisfy the ODE system
\begin{align*}
\ddot{f} & = - (- \dot{u} + \tr(L)) \dot{f} - \frac{\dot{f}^2}{f} 
+ \sum_{i=1}^m \frac{d_i q_i^2}{4} \frac{\omega_i^4}{f}, \\
\ddot{g}_i & = - (- \dot{u} + \tr(L)) \dot{g_i} - \frac{\dot{g}_i^2}{g_i} 
+ \frac{p_i - \frac{q_i^2}{2}\omega_i^2}{g_i}.
\end{align*}
Due to \eqref{APrioriBounds} and remark \ref{ConvergenceOmegaI}, the limit $\omega_{i, \infty} = \lim_{t \to \infty} \omega_i(t)$ exists for all $i.$ Moreover, the shape operator remains positive definite along the trajectories. Hence proposition \ref{AsymptoticsOfPotentialFunction} implies that $\frac{\dot{f}}{f},$ $\frac{\dot{g}_i}{g_i} \to 0$ and $-\dot{u} \to \sqrt{-C}$ as $t \to \infty.$  If $\omega_{i, \infty} > 0$ for some $i,$ then due to work of Appleton \cite[Lemma 6.2]{AppletonSteadyRS} it follows that for all small $\varepsilon > 0$ there is $t_0 >0$ such that
\begin{align}
f(t_0)^2 + \gamma_{-} (1+ \varepsilon)^{-1} (t-t_0) \leq & \ f^2(t) \leq f(t_0)^2 + \gamma_{+} (t - t_0), \nonumber \\
g_i(t_0)^2 + \Gamma_{i,-} (1+ \varepsilon)^{-1} (t-t_0) \leq & \ g_i^2(t) \leq g_i(t_0)^2 + \Gamma_{i,+} (t - t_0), \label{AsymptoticsGi}
\end{align}
where 
\begin{equation*}
 \gamma_{\pm} = \frac{\sum_{i=1}^m \frac{d_i q_i^2}{2} \omega_{i, \infty}^2 \pm \varepsilon}{\sqrt{-C} \mp \varepsilon}
 \ \text{and} \
\Gamma_{i, \pm} = \frac{2 p_i - q_i^2 \omega_{i, \infty}^2}{\sqrt{-C} \mp \varepsilon}.
\end{equation*}
This yields $\frac{\gamma_{-}}{\Gamma_{i,+}} \leq \omega_{i, \infty}^2 \leq \frac{\gamma_{+}}{\Gamma_{i,-}}$ for all small $\varepsilon >0,$ hence
\begin{align*}
\left( p_i - \frac{q_i^2}{2} \omega_{i, \infty}^2 \right) \omega_{i, \infty}^2 = \sum_{j=1}^m \frac{d_jq_j^2}{4} \omega_{j, \infty}^4
\end{align*}
and in particular $\omega_{i, \infty}>0$ for all $i.$ Notice that this implies
\begin{align*}
\sum_{j=1}^m d_j \omega_{j, \infty}^2 \left\lbrace  p_j - \left( \sum_{i=1}^m d_i + 2 \right) \frac{q_j^2}{4} \omega_{j, \infty}^2 \right\rbrace  = 0.
\end{align*}

In the domain $\lbrace \ \omega_{j, \infty}^2 \leq \frac{4p_j}{\left( \sum_{i=1}^m d_i + 2\right) q_j^2} \ \vert \ j=1, \ldots, m \ \rbrace $  the unique solution to this equation is $\omega_{j, \infty}^2 = \frac{4p_j}{\left( \sum_{i=1}^m d_i + 2 \right) q_j^2}$ for all $j.$ However, by assumption, the functions $\omega_i$ satisfy an even stronger bound. Thus $\omega_{i, \infty} = 0$ holds for those the trajectories.

The estimate \eqref{AsymptoticsGi} is still valid since $p_i >0$ and hence shows the asymptotic behaviour of all $g_i.$ It remains to prove that $f$ is bounded.

As $\dot{\omega}_i$ changes its sign at most once due to remark \ref{ConvergenceOmegaI}, there exists $t_0>0$ such that $\dot{\omega}_i(t) < 0$ for $t \geq t_0.$ Hence, for all $\delta > 0$ there exists $t_0>0$ such that 
\begin{align*}
\ddot{\omega}_i \leq - ( \sqrt{-C} + \delta ) \dot{\omega}_i - \left\lbrace  \frac{\sqrt{-C}}{2} - \delta \right\rbrace \frac{\omega_i}{t}
\end{align*}
for $t \geq t_0.$

As in \cite[Lemma 6.7]{AppletonSteadyRS}, it follows that for all $\delta > 0$ there exist $c_0,$ $t_0>0$ such that $\omega_i(t) \leq c_0 t^{-1/2+\delta}$ for $t \geq t_0.$ Hence there exist constants $c_1, c_2 >0$ such that $\ddot{f} \leq - c_1 \dot{f} + \frac{c_2}{t^{2-\delta}}$ for $t \geq t_0.$ This implies that $f$ is bounded and, as $\dot{f}>0,$ thus  converges.
\end{proof}

\begin{remarkroman} The K\"ahler condition $\frac{d}{dt} g_i^2 = - q_i f$ for the metric gives rise to the preserved locus of trajectories satisfying
\begin{align*}
\frac{\dot{g}_i}{g_i} = - \frac{q_i}{2} \frac{f}{g_i^2} \ \text{ and } \  \left(- \dot{u} + \tr(L) + \frac{\dot{f}}{f} \right) \frac{\dot{g}_i}{g_i} = \frac{\varepsilon}{2} + \frac{p_i}{g_i^2}.
\end{align*}

Due to the smoothness conditions of $f$ and $g_i,$  the condition $- q_i > 0$ is necessary for any K\"ahler metric to exist and then the shape operator is positive definite. In the steady case, proposition \ref{AsymptoticsOfPotentialFunction} furthermore implies $\frac{d}{dt} g_i^2 \to \frac{2 p_i}{\sqrt{-C}}$ as $t \to \infty$ and thus
\begin{align*}
f(t) \to - \frac{2 p_i}{\sqrt{-C} q_i} \ \text{ and } \ \frac{g_i^2(t)}{t} \to \frac{2 p_i}{\sqrt{-C}}
\end{align*}
as $t \to \infty.$

In particular, $\frac{p_i}{q_i}<0$ must be independent of $i$ for any K\"ahler metric to exist.
\end{remarkroman}

A modification of the argument in \cite{DWExpandingSolitons} shows that the expanding Ricci solitons are asymptotically conical:

\begin{proposition}
The expanding Ricci soliton metrics with \eqref{APrioriBounds} are complete and satisfy
\begin{align*}
\frac{- \dot{u}(t)}{t} \to \frac{\varepsilon}{2} \ \text{ and } \ t \cdot {L_t} \to \mathbb{I}_n
\end{align*}
as $t \to \infty.$
\label{AsymptoticsExpandingRSinProductsOverFanos}
\end{proposition}
\begin{proof}
Recall that \eqref{APrioriBounds} ensures that the shape operator $L_t$ remains positive definite. Thus, by definition, the variables $\mathcal{L}, X_i, Y_i, \omega_i$ of remark \ref{CompletenessSOneBundleViaBoundedPhaseSpaceRegion} remain positive and proposition \ref{AsymptoticsOfPotentialFunction} implies furthermore that $\mathcal{L}, X_i, Y_i \to 0$ as $s \to \infty.$ 

On the other hand, notice that the limits 
\begin{align*}
y_{i, \infty} = \lim_{s \to \infty} \frac{Y_i}{\mathcal{L}} = \lim_{t \to \infty} \frac{1}{g_i} \ \text{ and } \ 
\omega_{i, \infty}  = \lim_{s \to \infty} \omega_i 
\end{align*}
exist, see also remark \ref{ConvergenceOmegaI}, and furthermore $\omega_{i, \infty}^2 \leq \frac{2 p_i}{q_i^2}.$ The differential equations
\begin{align*}
\frac{d}{ds} \frac{X_0}{\mathcal{L}^2} & =\left( -\sum_{i=0}^m d_i X_i^2 -1 \right) \frac{X_0}{\mathcal{L}^2} + \sum_{i=1}^m \frac{d_i q_i^2}{4} \omega_i^2 \frac{Y_i^2}{\mathcal{L}^2} + \frac{\varepsilon}{2} \left( X_0 + 1 \right), \\
\frac{d}{ds} \frac{X_i}{\mathcal{L}^2} & =\left( -\sum_{i=0}^m d_i X_i^2 -1 \right) \frac{X_i}{\mathcal{L}^2} 
+ \left( p_i - \frac{q_i^2}{2} \omega_i^2 \right)  \frac{Y_i^2}{\mathcal{L}^2} + \frac{\varepsilon}{2} \left( X_i + 1 \right)
\end{align*}
follow directly from remark \ref{CompletenessSOneBundleViaBoundedPhaseSpaceRegion} and imply that 
\begin{align*}
\frac{X_0}{\mathcal{L}^2} & \to \Lambda_0 = \sum_{i=1}^m \frac{d_i q_i^2}{4} \omega_{i, \infty}^2 y_{i, \infty}^2 + \frac{\varepsilon}{2} > 0 \\
\frac{X_i}{\mathcal{L}^2} & \to \Lambda_i = \left( p_i - \frac{q_i^2}{2} \omega_{i, \infty}^2 \right) y_{i, \infty}^2 + \frac{\varepsilon}{2} > 0
\end{align*}
as $s \to \infty.$ Thus 
\begin{align*}
\frac{Y_i^{'}}{\mathcal{L}^{'}} = \frac{Y_i \left( \sum_{i=0}^m d_i X_i^2 - \frac{\varepsilon}{2} \mathcal{L}^2 - X_i \right)}{\mathcal{L} \left( \sum_{i=0}^m d_i X_i^2 - \frac{\varepsilon}{2} \mathcal{L}^2 \right)} \to y_{i, \infty} \cdot \frac{\varepsilon + 2 \Lambda_i}{\varepsilon}
\end{align*}
as $s \to \infty$. If $y_{i, \infty} >0$, then L'H\^{o}pital's rule implies the contradiction $\Lambda_i = 0.$ Thus $y_{i, \infty} =0$ and $\Lambda_i = \frac{\varepsilon}{2}$ for $i = 0, \ldots, m.$ 

Having established $\frac{X_i}{\mathcal{L}^2} \to \frac{\varepsilon}{2}$ as $s \to \infty,$ integrating the ODE for $\mathcal{L}$ as in \cite[discussion after Lemma 3.15]{DWExpandingSolitons} yields $\mathcal{L}^2 \cdot s \to \frac{1}{\varepsilon}$ as $s \to \infty.$ The relation $dt = \mathcal{L} ds$ thus implies $s \sim \frac{\varepsilon}{4} t^2$ and hence $\mathcal{L} \cdot t \to \frac{2}{\varepsilon}$ as $t \to \infty.$

According to the definition of the coordinate change in remark \ref{CompletenessSOneBundleViaBoundedPhaseSpaceRegion}, the soliton potential is recovered via $-\dot{u} = \frac{1}{\mathcal{L}} \left( 1 - \sum_{i=0}^m d_i X_i \right)$ and the shape operator is determined by the quotients $\frac{X_i}{\mathcal{L}}$. The corresponding asymptotic behaviour is immediate.
\end{proof}

\subsubsection{The L\"u-Page-Pope set-up}
\label{SectionLuPagePopeSetUp}

Let $L$ be the total space of a complex line bundle over a single Fano K\"ahler Einstein manifold whose Euler class is a rational multiple of the first Chern class of the base, equipped with the same geometry as in section \ref{DancerWangSetUpSection}. L\"u-Page-Pope \cite{LuPagePopeQEinstein} have found Ricci flat metrics of warped product type on $L \times S^{m},$ for $m > 1,$ by exhibiting an explicit formula. This section shows that all of their examples also support complete steady and expanding Ricci solitons, thus proving Theorem \ref{RSInLuPagePopeSetUp}. In fact, as L\"u-Page-Pope remark, the factor $S^{m}$ can be replaced by any Einstein manifold $(N^{d_2},h_2)$ of positive Ricci curvature. Suppose that $\Ric_N = (d_2-1) h_2.$ Then, with the same notation for the metric on the principal circle bundle as in section \ref{DancerWangSetUpSection}, the metric on the principal hypersurface is given by the warped product
\begin{equation*}
g_t = f(t)^2 \theta \otimes \theta + g_1(t)^2 \pi^{*} h_1 + g_2(t)^2 h_2
\end{equation*}
and the Ricci soliton equations are given by the ODE system
\begin{align*}
\ddot{u} = & \ \frac{\ddot{f}}{f} + \sum_{i=1}^2 \frac{\ddot{g}_i}{g_i} - \frac{\varepsilon}{2}, \\
\frac{d}{dt} \frac{\dot{f}}{f} = & \ - \frac{\dot{f}}{f} (- \dot{u} + \tr(L) ) + \frac{\varepsilon}{2} + d_1 \frac{q_1^2}{4} \frac{f^2}{g_1^4}, \\
\frac{d}{dt} \frac{\dot{g}_1}{g_1} = & \ - \frac{\dot{g}_1}{g_1} (-\dot{u} + \tr(L)) + \frac{\varepsilon}{2} + \frac{p_1}{g_1^2} - \frac{q_1^2}{2} \frac{f^2}{g_1^4}, \\
\frac{d}{dt} \frac{\dot{g}_2}{g_2} = & \ - \frac{\dot{g}_2}{g_2} (-\dot{u} + \tr(L)) + \frac{\varepsilon}{2} + \frac{d_2-1}{g_2^2}.
\end{align*}
This ansatz gives rise to smooth Ricci soliton metrics provided that the initial conditions 
\begin{align}
f(0)=0, & \ \dot{f}(0)=1 \ \text{ and } \ g_i(0) >0, \ \dot{g}_{i}(0) = 0, \label{LuPagePopeInitialConditions} \\
 & \hspace{-0.6mm} u(0)=\dot{u}(0)=0, \ 2 \ddot{u}(0) = C \nonumber
\end{align}
are satisfied, where $C$ is the constant in the conservation law \eqref{GeneralConservationLaw}. 

Notice that this is a special case of the Dancer-Wang set-up in section \ref{DancerWangSetUpSection} if one chooses $p_2=d_2-1$ and $q_2=0.$ However, this allows a simplified argument, which shall briefly be sketched: 

Fix $\varepsilon \geq 0.$ As in section \ref{DancerWangSetUpSection} one has $\ddot{g}_1(0)>0$ and the assumption $d_2>1$ implies that also $2 \frac{\ddot{g}_2(0)}{g_2(0)} = \frac{\varepsilon}{2} + \frac{d_2-1}{g_2^2(0)} >0.$ In particular it follows that $\dot{g}_i(t)>0$ for $i=1,2$ and sufficiently small $t>0.$ Furthermore it is clear from the ODE system that $\dot{f}>0$ and $\dot{g_2}>0$ are preserved. Due to proposition \ref{CompletenessDancerWangSetUp}, in order to establish completeness, it suffices to show that the set 
\begin{align*}
\left\lbrace \ \dot{g}_1 > 0 \ \text{ and } \ \frac{f^2}{g_1^2} < \frac{4p_1}{(d_1+2)q_1^2} \ \right\rbrace 
\end{align*}
is preserved by the Ricci soliton ODE if $\varepsilon \geq 0.$ This is now straightforward:

Consider the quotient $\omega_1 = \frac{f}{g_1}.$ It satisfies the ODE
\begin{align*}
\ddot{\omega}_1 = \dot{\omega}_1 \cdot \frac{d}{dt} \left\lbrace  u - \ln g_1^{d_1+1} g_2^{d_2} \right\rbrace + \frac{\omega_1}{g_1^2} \left\lbrace \frac{(d_1+2)q_1^2}{4} \omega_1^2 - p_1 \right\rbrace.
\end{align*}
In particular, as long as $\omega_1^2(t) < \frac{4p_1}{(d_1+2)q_1^2},$ one has 
\begin{align*}
\dot{\omega}_1(t) \leq \frac{(g_1^{d_1}g_2^{d_2})(0)}{(g_1^{d_1+1}g_2^{d_2})(t)} \exp(u(t)) \leq \frac{1}{g_1(0)}
\end{align*}
as $u(0)=0$ and $\dot{\omega}_1(0)= \frac{1}{g_1(0)}.$ Then the same procedure as in the proof of proposition \ref{CompletenessDimOneIsOneWithPotentialFunctionTrick} yields Theorem \ref{RSInLuPagePopeSetUp}:

\begin{proposition} Let $d_2>1$ and $\varepsilon\geq 0.$ Then for any $\bar{g}>0$ there exists a constant $C_0 \leq 0$ such that the steady and expanding Ricci soliton trajectories in the L\"u-Page-Pope set-up with $C < C_0$ and $g_i(0)>\bar{g}$ for $i=1,2$ in \eqref{LuPagePopeInitialConditions} and are complete.
\label{CompleteTrajectoriesLuPagePopeSetUp}
\end{proposition}

\begin{remarkroman}
It is easy to check that the proof of proposition \ref{CompleteTrajectoriesLuPagePopeSetUp} also works if $\varepsilon>0$ and $d_2=1.$ Moreover, in this case, it is sufficient to assume $g_2(0)>0.$
\end{remarkroman}


\end{document}